\documentclass[11pt,hidelinks]{article}

\title{A general approach to massive upper bound for two-point function with application to self-avoiding walk torus plateau}

\author{
Yucheng Liu\,\orcidlink{0000-0002-1917-8330}\thanks{Department of Mathematics,
	University of British Columbia,
	Vancouver, BC, Canada V6T 1Z2.
	\href{mailto:yliu135@math.ubc.ca}{yliu135@math.ubc.ca}.
	}
}
\date{\vspace{-5ex}} 

\usepackage{orcidlink}
\usepackage{amsmath, amssymb, amscd, amsthm, amsfonts}
\usepackage{graphicx} 
\usepackage{hyperref} 
\usepackage{xcolor}
\usepackage{dsfont, bm}
\usepackage[title]{appendix}
\usepackage{comment}
\usepackage{enumerate}
\usepackage{enumitem}
\usepackage{cite}

\usepackage[textwidth=480pt,textheight=650pt,centering]{geometry} 

\usepackage{tikz}
\tikzset{every picture/.style={line width=0.75pt}}

\theoremstyle{plain}
\newtheorem{theorem}{Theorem}[section]
\newtheorem{lemma}[theorem]{Lemma}

\newtheorem{proposition}[theorem]{Proposition}
\newtheorem{corollary}[theorem]{Corollary}

\newtheorem{assumption}[theorem]{Assumption}

\numberwithin{equation}{section}

\newcommand{\ie}{i.e.}
\newcommand{\eg}{e.g.}

\newcommand{\eps}{\varepsilon}

\newcommand{\N}{\mathbb{N}}
\newcommand{\Z}{\mathbb{Z}}

\newcommand{\R}{\mathbb{R}}
\newcommand{\C}{\mathbb{C}}

\newcommand{\T}{\mathbb{T}}

\newcommand{\Scal}{\mathcal{S}}
\newcommand{\Wcal}{\mathcal{W}}

\newcommand{\del}{\partial}
\newcommand{\grad}{\nabla}
\newcommand{\inv}{^{-1}}
\renewcommand{\(}{\left(}
\renewcommand{\)}{\right)}
\newcommand{\half}{\frac{1}{2}}

\newcommand{\1}{\mathds{1}}

\newcommand{\nl}{\nonumber \\}

\renewcommand{\Re}{\mathrm{Re}\,}
\renewcommand{\Im}{\mathrm{Im}\,}

\providecommand{\abs}[1]{\lvert#1\rvert}
\providecommand{\norm}[1]{\lVert#1\rVert}

\providecommand{\bigabs}[1]{\big\lvert#1\big\rvert}
\providecommand{\Bigabs}[1]{\Big\lvert#1\Big\rvert}

\providecommand{\bignorm}[1]{\big\lVert#1\big\rVert}

\providecommand{\biggnorm}[1]{\bigg\lVert#1\bigg\rVert}

\newcommand{\mz}{m(z)}
\newcommand{\muz}{{\mu_z}}
\newcommand{\supm}{^{(m)}}
\newcommand{\supM}{^{(M)}}
\newcommand{\supN}{^{(N)}}
\newcommand{\supNm}{^{(N,m)}}
\newcommand{\supT}{^{ \mathbb{T} }}
\newcommand{\subrz}{_{r,z}}
\newcommand{\rz}{{r,z}}
\newcommand{\st}{{st}}

\newcommand{\veee}[1]{|\!|\!|#1|\!|\!|}
\newcommand{\xvee}{\veee{x}}
\providecommand{\nnnorm}[1]{\veee {#1} }
\newcommand{\const}{\mathrm{const}}

\newcommand{\Dnn}{D}
\newcommand{\crit}{_{z_c}}

\newcommand{\PiFour}{
\begin{tikzpicture}[x=0.75pt,y=0.75pt,yscale=-1,xscale=1]

\draw    (140,120) -- (180,120) ;
\filldraw[fill=white] (140,120) circle (0pt) node[left]{$0$};
\draw    (140,120) -- (160.4,85.4) ;
\draw    (160.4,85.4) -- (200.4,85.4) ;
\filldraw[fill=white] (200.4,85.4) circle (0pt) node[right]{$x$};
\draw    (160.4,85.4) -- (180,120) ;
\draw    (180,120) -- (200.4,85.4) ;
\draw  [draw opacity=0] (140,119.84) .. controls (140.06,105) and (148.19,92.07) .. (160.24,85.21) -- (180,120) -- cycle ; \draw   (140,119.84) .. controls (140.06,105) and (148.19,92.07) .. (160.24,85.21) ;  
\draw  [draw opacity=0] (200.2,85.34) .. controls (200.16,100.31) and (191.9,113.34) .. (179.71,120.16) -- (160.2,85.24) -- cycle ; \draw   (200.2,85.34) .. controls (200.16,100.31) and (191.9,113.34) .. (179.71,120.16) ;  
\draw    (183.2,81.8) -- (178.8,89.4) ;
\draw    (162.8,115.4) -- (158.4,123) ;
\end{tikzpicture}
}

\begin{document}
\maketitle

\begin{abstract}
We prove a sufficient condition for the two-point function of a statistical mechanical model on $\mathbb{Z}^d$, $d > 2$, to be 
bounded uniformly near a critical point by $|x|^{-(d-2)} \exp [ -c|x| / \xi ]$, where $\xi$ is the correlation length.
The condition is given in terms of a convolution equation satisfied by the two-point function, and we verify the condition for strictly self-avoiding walk in dimensions $d > 4$ using the lace expansion. 
As an example application, 
we use the uniform bound to study the self-avoiding walk on a $d$-dimensional discrete torus with $d > 4$, proving a ``plateau'' of the torus two-point function, a result previously obtained for weakly self-avoiding walk in dimensions $d > 4$ by Slade. 
Our method has the potential to be applied to other statistical mechanical models on $\mathbb{Z}^d$ or on the torus. 
\end{abstract}




\section{Introduction and results}

\subsection{Introduction}
We consider the two-point function $G_z$ of a statistical mechanical model on $\Z^d$, $d > 2$, near a critical point $z_c$ from the subcritical regime. The parameter $z$ here can be, \eg, inverse temperature of the Ising model, bond density of Bernoulli bond percolation, or fugacity of the self-avoiding walk. 
Our main result is a general sufficient condition for the uniform upper bound
\begin{align}
\label{eq:scaling_ub}
G_z(x) \le \frac{ c_0 } { \max \{ 1,  \abs x^{d-2} \} } e^{-c_1 \abs x / \xi(z)}
	\qquad (z\approx z_c,\ x\in \Z^d),
\end{align}
where $\xi(z)$ is the correlation length of the model,
to hold with some constants $c_0 > 0$ and $c_1 \in (0,1)$. 

Uniform bounds of the form \eqref{eq:scaling_ub} have been obtained recently, for weakly self-avoiding walk in dimensions $d>4$ \cite{Slad23_wsaw}, for percolation in $d>6$ \cite{HMS23},
for the Ising model in $d>4$ \cite{DP25-Ising}, 
and for lattice trees and lattice animals in $d>8$ \cite{LS25a}.
The bounds are useful in studying the models on the discrete torus (a box with periodic boundary conditions).
For example, they can be used to prove the ``plateau'' of the torus two-point function at the infinite-volume critical point \cite{LPS25-universal}, which, roughly speaking, is the phenomenon that the torus two-point function decays like its $\Z^d$ counterpart for small $x$, and stops decaying and levels off to a nonzero constant for larger $x$ in the torus. 
For the weakly self-avoiding walk, the plateau and the upper bound \eqref{eq:scaling_ub} are then used in \cite{MS23, Mich23} to prove that the walk on a torus of volume $V$ needs to have length at least $\sqrt V$ to ``feel'' that it lives on the torus. 
For percolation, it is shown in \cite{HMS23} that the plateau and the upper bound \eqref{eq:scaling_ub} allow a proof of the torus triangle condition (defined in \cite{BCHSS05a}) without the torus lace expansion of \cite{BCHSS05b}. 
The results of \cite{MS23, Mich23, HMS23, LPS25-universal} provide the new perspective that statistical mechanical models on a torus can be studied using results on $\Z^d$. 
With this perspective, we expect our general upper bound to be useful for many statistical mechanical models on the torus. 
As an illustration, we prove the ``plateau'' phenomenon for strictly self-avoiding walk in dimensions $d>4$. 

Our sufficient condition for \eqref{eq:scaling_ub} is given in terms of a convolution equation satisfied by $G_z$.
For self-avoiding walk in dimensions $d>4$, the convolution equation is provided by Brydges and Spencer's lace expansion \cite{BS85, Slad06}. 
Since there are versions of the lace expansion derived for many statistical mechanical models, \eg,
the Ising model \cite{Saka07} and $\varphi^4$ model \cite{BHH21, Saka15} in dimensions $d > 4$,
percolation in $d > 6$ \cite{HS90a, FH17}, 
lattice trees and lattice animals in $d > 8$ \cite{HS90b,FH21}, 
we expect our result to be applicable to all these models.
In fact, the analysis of lattice trees/animals in \cite{LS25a} uses some ideas developed in this paper.

\medskip\noindent
{\bf Notation:}  
We write $a \vee b = \max \{ a , b \}$ and $a \wedge b = \min \{ a , b \}$.
We write $f = O(g)$ or $f \lesssim g$ to mean there is a constant $C> 0$ for which $\abs {f(x)} \le C\abs {g(x)}$, 
and write $f \asymp g$ to mean $g \lesssim f \lesssim g$.
To avoid dividing by zero, with $|x|$ the Euclidean norm of $x\in \R^d$, we define
\begin{equation}
\label{eq:xvee1}
	\xvee = \max\{|x| , 1\}.
\end{equation}
Note that \eqref{eq:xvee1} does not define a norm on $\R^d$.

\subsection{Main result}

To motivate our sufficient condition, we first recall the simple random walk.
Let $\Dnn: \Z^d \to [0,1)$ be defined by $\Dnn(x) = \frac{1}{2d} \1\{|x|=1\}$, and let $\abs\Omega = 2d$. 
For $\mu \in (0,\abs\Omega\inv]$, the \emph{random walk two-point function} (lattice Green function)
$C_\mu : \Z^d \to \R$ is the solution, vanishing at infinity, of the convolution equation
\begin{equation}
\label{eq:LGF}
    (\delta -\mu\abs\Omega \Dnn)*C_\mu = \delta ,
\end{equation}
where $\delta$ is the Kronecker delta $\delta(x)=\delta_{0,x} = \1\{x=0\}$,
and the convolution is defined by $(f*g)(x) = \sum_{y\in \Z^d} f(y)g(x-y)$.
In dimensions $d > 2$,
equation \eqref{eq:LGF} can be solved using the (inverse) Fourier transform
\begin{equation}
\hat f(k)  = \sum_{x\in\Z^d}f(x) e^{ik\cdot x} 	\quad (k \in \T^d),
\qquad
f(x) = \int_{\T^d} \hat f(k) e^{-ik\cdot x}  \frac{ dk }{ (2\pi)^d } \quad (x\in \Z^d),
\end{equation}
where $\T^d=(\R/2\pi\Z)^d$ is the continuum torus, which we identify with $(-\pi,\pi]^d \subset \R^d$.
That is, $C_\mu$ is given by the absolutely convergent integral
\begin{equation}
C_\mu(x) 
= \int_{\T^d} \frac{ e^{-ik\cdot x} }{1 - \mu\abs\Omega\hat \Dnn(k)} \frac{ dk }{ (2\pi)^d },
	\qquad  \hat \Dnn(k) = d^{-1}\sum_{j=1}^d \cos k_j.
\end{equation}
It is known that $C_\mu$ satisfies a near-critical estimate of the form \eqref{eq:scaling_ub} (see Lemma~\ref{lem:C}).
At the critical point $\mu_c = 1/(2d)$, we have $C_{1/(2d)}(x) \sim \const\, \abs x^{-(d-2)}$ as $\abs x\to \infty$ (see, e.g., \cite{LL10});
detailed asymptotics of subcritical $C_\mu$ are given in \cite{MS22}.

In analogy to the random walk, and motivated by applications based on the lace expansion, we assume the two-point function $G_z$ of interest satisfies a convolution equation
\begin{equation} \label{eq:FGz}
F_z * G_z = \delta,
\end{equation}
with $F_z$ satisfying assumptions stated below, 
and we assume $G_z$ is given by the Fourier integral
\begin{equation} \label{eq:Gint}
G_z(x) = \int_{\T^d} \frac{ e^{-ik\cdot x} }{ \hat F_z(k) } \frac{dk}{(2\pi)^d}.
\end{equation}
Note that $G_z$ need not come from a nearest-neighbour model.
The functions $F_z, G_z$ typically decay exponentially in the subcritical regime. In our assumptions, we remove part of the exponential decay using an \emph{exponential tilting}: 
Given $m\ge 0$ and a function $f : \Z^d \to \R$, we write
\begin{equation}
f\supm(x) = f(x)e^{mx_1} 	\qquad (x = (x_1, \dots, x_d)),
\end{equation}
and we write $\hat f \supm(k) = \sum_{x\in \Z^d} f\supm(x) e^{ik\cdot x}$ for its Fourier transform.

\begin{assumption}
\label{ass:Fz}
We assume there is a family of $\Z^d$-symmetric functions 
(invariant under reflection in coordinate hyperplanes and rotation by $\pi/2$)
$F_z : \Z^d \to \R$, 
defined for $z \in [z_c - \delta, z_c)$ with some $z_c, \delta >0$, 
and there is a function $m: [z_c - \delta, z_c ) \to (0, \infty)$, 
such that:
\begin{enumerate}[label=(\roman*)]

\item 
There are constants $\eps > 0$, $K_1 > 0$ such that $\sum_{x\in \Z^d} \abs x^{2+\eps} \abs{ F_z\supm (x) } \le K_1$ uniformly in $z$ and in $m \in [0, \mz)$;

\item
$\hat F_z\supm (0) \ge 0$ for all $z$ and all $m \in [0, \mz)$,
and $\hat F_z(0) \to 0$ as $z\to z_c$;

\item \emph{(Untilted infrared bound)}
There is a constant $K_2 >0$ such that
$\hat F_z(k) - \hat F_z(0) \ge K_2 \abs k^2$ uniformly in $k\in \T^d$ and in $z$;

\item 
$\mz \to 0$ as $z\to z_c$.
\end{enumerate}
If $d > 4$, we further assume 
\begin{enumerate}[label=(\roman*)] \setcounter{enumi}{4}
\item
For some $1 \le p < d/4$, we have
\[
\sup_z \sup_{0\le m < \mz}
\bignorm{ \abs x^{d-2} F_z\supm (x) }_{\ell^{p\wedge 2}(\Z^d)} < \infty .
\]
\end{enumerate}
\end{assumption}

For the self-avoiding walk, we will take the function $\mz$ to be the exponential decay rate of $G_z$ (\ie, the inverse correlation length $\xi(z)\inv$), see \eqref{eq:def_mass}.
But $\mz$ can be chosen in other ways, as in \cite[Section~1.3]{LS25a}.
Assumption~\ref{ass:Fz}(ii) and (iv) are indications of a subcritical model near the critical point $z_c$. 
Assumption~\ref{ass:Fz}(iii) and $\hat F_z(0) \ge 0$ imply absolute convergence of the Fourier integral \eqref{eq:Gint} since $d>2$. 
Assumption~\ref{ass:Fz}(i) and (v) are implied by the stronger assertion that
\begin{equation} \label{eq:F_decay}
\sup_z \sup_{0\le m < \mz}
\abs{ F_z\supm (x) } \lesssim \frac 1 { \nnnorm x^{d + 2 + \rho} }
\quad \text{with}\quad 
\rho > \frac { d-8 } 2 \vee 0 .
\end{equation}
The critical, untilted ($m=0$) case of \eqref{eq:F_decay} and Assumption~\ref{ass:Fz}(ii)--(iii) are used in \cite{LS24a} to prove that $G\crit(x) \sim \const\, \abs x^{-(d-2)}$ as $\abs x\to \infty$;
their proof also gives the uniform upper bound $G_z (x) \lesssim \nnnorm x^{-(d-2)}$ for near-critical $z$. 
We improve the bound by an exponential factor. 

\begin{theorem}
\label{thm:near_critical}
Let $d>2$, let $F_z$ satisfy Assumption~\ref{ass:Fz}, and 
let $G_z$ be the solution to \eqref{eq:FGz} given by the Fourier integral \eqref{eq:Gint}.
Then there are constants $c_0 > 0$, $c_1 \in (0,1)$, and $\delta>0$ such that for all $z \in [z_c - \delta, z_c)$ and $x\in \Z^d$, 
\begin{equation} \label{eq:G_ub}
G_z(x) \le \frac{ c_0 } { \nnnorm x^{d-2} } e^{-c_1 m(z) \abs x} .
\end{equation}
\end{theorem}

The conclusion of Theorem~\ref{thm:near_critical} is \eqref{eq:scaling_ub} when $\mz$ is taken to be $\xi(z)\inv$.
The theorem is inspired by both the analysis of weakly self-avoiding walk in \cite{Slad23_wsaw} and the analysis of general convolution equations in \cite{LS24a}. Both references are based on \cite{Slad22_lace}. 
We follow \cite{LS24a} in the use of \emph{weak derivatives}, which overcomes the obstacle identified in \cite[Remark~5.3]{Slad23_wsaw} and opens up the possibility for handling percolation and lattice trees/animals within the same framework; previously, the upper bound \eqref{eq:G_ub} for high-dimensional percolation is proved using methods specific to percolation that does not apply to self-avoiding walk \cite{HMS23}. 
For spread-out percolation in dimensions $d >  6$, 
its $F_z$ is proved in \cite{LS24b} (using results of \cite{HHS03}) to satisfies the decay estimate \eqref{eq:F_decay} with $\rho = d-6$ when $m=0$.
The decay estimate is expected to hold for all $m \in [0, \mz)$, by extending diagrammatic estimates in \cite{HHS03}.

\subsection{Applications to self-avoiding walk}
\label{sec:saw_intro}

\subsubsection{The model}
The self-avoiding walk is defined as follows. 
Recall $D(x) = \frac 1 {2d}  \1\{\abs x=1\}$.
For $n\in \N$, we use $D^{*n}$ to denote the $n$-fold convolution of $D$ with itself, and we define $D^{*0}(x) = \delta_{0,x}$. 
We let $\mathcal W_n(x)$ denote the set of $n$-step walks form 0 to $x$, \ie, the set of $\omega = (\omega(0), \dots, \omega(n))$ with $\omega(i) \in \Z^d$ for all $i$, $\omega(0)=0$, $\omega(n)=x$, and $ \abs{ \omega(i) - \omega(i-1) } = 1$ for all $i\ge 1$,
and let $\mathcal S_n(x)$ denote the set of $n$-step \emph{self-avoiding walks} from 0 to $x$, which are $n$-step walks from 0 to $x$ that satisfies $\omega(i) \ne \omega(j)$ for all $i\ne j$. 
For $n=0$, we let $\mathcal W_0(x) = \mathcal S_0(x)$ consist of the zero-step walk $\omega(0)=0$ when $x=0$, and we let it be the empty set otherwise. 
Given $z\ge 0$ and $x\in \Z^d$, we define the \emph{self-avoiding walk two-point function} to be
\begin{equation} \label{eq:def_G}
G_z(x) = \sum_{n=0}^\infty \sum_{\omega \in \Scal_n(x) } z^n.
\end{equation}
We note that the random walk two-point function defined in \eqref{eq:LGF} has the explicit formula
\begin{equation} \label{eq:def_C}
C_\mu(x) = \sum_{n=0}^\infty  (\mu \abs\Omega D)^{*n}(x)
= \sum_{n=0}^\infty \sum_{\omega \in \Wcal_n(x) } \mu^n. 
\end{equation}
Compared to \eqref{eq:def_C}, the sum defining $G_z$ sums over only the walks with no self-intersections. 
A standard subadditivity argument implies the existence of $z_c \ge \mu_c = \abs\Omega\inv$ such that the \emph{susceptibility} $\chi(z) = \sum_{x\in\Z^d} G_z(x)$ is finite if and only if $z \in [0,z_c)$, so the sum defining $G_z(x)$ converges at least for $z \in [0, z_c)$. 
It also follows from subadditivity that $\chi(z) \ge (1- z/z_c)\inv$ (see, e.g., \cite[Theorem~2.3]{Slad06}).

For all dimensions $d \ge 2$, it is known (\eg, \cite[Section~4.1]{MS93}) that $G_z(x)$ decays exponentially when $z \in (0,z_c)$, with \emph{mass} (\ie, exponential decay rate)
\begin{equation} \label{eq:def_mass}
m(z) = \lim_{\abs x_z \to \infty} \frac{ - \log G_z(x) }{ \abs x_z } \in (0,\infty), 
\end{equation}
where $\abs \cdot_z$ is a $z$-dependent norm satisfying $\norm x_\infty \le \abs x_z \le \norm x_1$. 
The \emph{correlation length} is $\xi(z) = \mz\inv$. 
It is also known that $m(z)$ is continuous and strictly decreasing in $z$, with limits $\mz \to \infty$ as $z\to 0$ and $\mz \to 0$ as $z\to z_c$. 
The two-point function $G_z$ satisfies the bound \cite[Lemma~4.1.5]{MS93}
\begin{equation} \label{eq:G_exp_bound}
G_z(x) \le \norm{ G_z }_{\ell^2(\Z^d)} \, e^{-m(z) \norm x_\infty}
	\qquad (x\in \Z^d)
\end{equation}
and obeys the Ornstein--Zernike decay \cite{CC86b}, \cite[Theorem~4.4.7]{MS93}
\begin{equation} \label{eq:OZ}
G_z(x_1, 0, \dots, 0) \sim \frac{ c_z }{ x_1^{(d-1)/2} } e^{-m(z) x_1}
	\qquad (x_1\to \infty)
\end{equation}
with some $c_z >0$. 

For critical $z_c$, it is partially proven that 
\begin{equation} \label{eq:G_crit_asymp}
G\crit(x) \sim  \frac \const {\abs x^{d-2+\eta} }
\qquad (\abs x\to \infty),
\end{equation}
where $\eta$ is a universal critical exponent.
The conjectured values for $\eta = \eta(d)$ are $\eta(2) = 5/24$, $\eta(3) \approx 0.031043$, and $\eta(4) = 0$.\footnote{
$\eta(4)=0$ is proved for a related model called the continuous-time weakly self-avoiding walk \cite{BBS-saw4}, using a rigorous renormalisation group method.
} 
In dimensions $d>4$, \eqref{eq:G_crit_asymp} is proved with $\eta = 0$ by \cite{Hara08}, using the lace expansion; the proof is simplified in \cite{LS24a}. 
The lace expansion also yields many other \emph{mean-field} results for self-avoiding walk in dimensions $d>4$, including the following asymptotic behaviours of the susceptibility and the mass as $z\to z_c$ \cite{HS92a} 
\begin{equation} \label{eq:mass_asymp}
\chi(z) \sim \const (z_c - z)\inv, \qquad 
\mz \sim \const (z_c - z)^{1/2}.
\end{equation}
More background materials on self-avoiding walk can be found in \cite{MS93, Slad06, Liu25_thesis}. 

Motivated by the behaviours of the simple random walk \cite{MS22},
we believe the crossover from the critical behaviour \eqref{eq:G_crit_asymp} to the subcritical Ornstein--Zernike behaviour \eqref{eq:OZ} to take the form
\begin{equation}
G_z(x) \approx \frac {\const}{ \abs x_z^{d-2} } \times \begin{cases}
1 & (\abs x_z \le \mz\inv ) \\
( \mz \abs x_z )^{(d-3)/2} e^{-\mz \abs x_z} &(\abs x_z \ge \mz \inv)
\end{cases}
\end{equation}
in dimensions $d > 4$.
Our first result is an upper bound in this direction.

\subsubsection{Self-avoiding walk results}
Using Theorem~\ref{thm:near_critical}, with the convolution equation \eqref{eq:FGz} provided by the lace expansion, we prove the following theorem. 

\begin{theorem}
\label{thm:saw}
Let $d > 4$. 
There are constants $c_0 > 0$ and $c_1 \in (0,1)$ such that for all $z \in (0, z_c)$ and $x\in \Z^d$, 
\begin{equation} \label{eq:G_saw}
G_z(x) \le \frac{ c_0 } { \nnnorm x^{d-2} } e^{-c_1 m(z) \abs x} .
\end{equation}
\end{theorem}

The conclusion of Theorem~\ref{thm:saw} trades some of the exponential decay in \eqref{eq:G_exp_bound} into a power law prefactor $\abs x^{-(d-2)}$. In light of the Ornstein--Zernike decay \eqref{eq:OZ}, it is impossible to have $c_1 = 1$ in \eqref{eq:G_saw} unless $d=3$. 

The next result concerns self-avoiding walks on the discrete torus $\T^d_r = (\Z / r\Z)^d$. 
Its statement involves the evaluation of $G_z$ at a torus point $x$, through the identification of $\T^d_r$ with $[ - \frac r 2, \frac r 2 )^d \cap \Z^d$.
The theorem shows that the slightly subcritical torus two-point function $G_z\supT(x)$ (defined by the analogue of \eqref{eq:def_G} using walks on $\T^d_r$ rather than on $\Z^d$) behaves like its $\Z^d$ counterpart $G_z(x)$, plus a constant term that can dominate when $\abs x$ is large. 
The constant term is the torus ``plateau.'' 

\begin{theorem}
\label{thm:torus}
Let $d>4$. 
There are constants $c_i>0$ and $M > 0$ such that 
\begin{equation} \label{eq:G^T}
G_z(x) + c_1 \frac{ \chi(z) }{ r^d }
\le G_z\supT(x)
\le G_z(x) + c_2 \frac{ \chi(z) }{ r^d } e^{-c_5 m(z) r}, 
\end{equation}
where the upper bound holds for all $r\ge3$, all $x\in \T^d_r$, and all $z\in (0,z_c)$, 
whereas the lower bound holds for all $r\ge 3$, all $x\in \T^d_r$ with $\abs x \ge M$, and all $z\in [ z_c - c_3 r^{-2}, z_c - c_4 r^{-d/2} ]$. 
\end{theorem}

Note the interval for $z$ is nonempty when $r$ is sufficiently large since $d>4$. 
At $z=z_c$, since $\chi(z_c) = \infty$ and $G\supT\crit(x)$ is finite ($G_z\supT(x)$ is a polynomial of degree at most $V = r^d$ in $z$), the upper bound of \eqref{eq:G^T} holds trivially and the lower bound cannot hold. 
Monte Carlo verification of \eqref{eq:G^T} was carried out in \cite{DGGZ24,ZGFDG18}. 
The next corollary, although less precise, shows more clearly that $G_z\supT(x)$ stops decaying and levels off when $x$ is large. 

\begin{corollary} \label{cor:torus}
Let $d>4$. There are constants $c_3, c_4 >0$ and $r_0 \ge 3$ such that
\begin{equation} \label{eq:G_cor}
G_z\supT(x) \asymp \frac 1 { \nnnorm x^{d-2} } + \frac{ \chi(z) }{ r^d }
\end{equation}
for all $r\ge r_0$, all $x\in \T^d_r$, and all $z\in [ z_c - c_3 r^{-2}, z_c - c_4 r^{-d/2} ]$. 
In particular, 
\begin{equation}
G\crit \supT(x) 
\ge G_{z_c - c_4 r^{-d/2}}\supT(x)
\gtrsim \frac 1 { \nnnorm x^{d-2} } + \frac{ 1 }{ r^{d/2} } .
\end{equation}
\end{corollary} 

The constant $r^{-d/2}$ is believed to be of the correct order for the plateau at $z_c$. No plateau is predicted for the model below the upper critical dimension, nor for the model with free boundary conditions at $z_c$.
We refer to \cite{LPS25-universal, Slad23_wsaw, MPS23} and references therein for more discussions.

\subsection{Organisation}
In Section~\ref{sec:near_critical} we prove Theorem~\ref{thm:near_critical}. The proof uses basic Fourier analysis, product and quotient rules of differentiation, and H\"older's inequality. 
In Section~\ref{sec:saw}, we first recall some lace expansion results from previous works, and then we use them to verify the hypotheses of Theorem~\ref{thm:near_critical}, to prove Theorem~\ref{thm:saw}.
In Section~\ref{sec:torus}, we use a technique called ``unfolding'' that relates torus walks to $\Z^d$ walks, together with Theorem~\ref{thm:saw}, to prove Theorem~\ref{thm:torus} and Corollary~\ref{cor:torus}.

\section{Proof of Theorem~\ref{thm:near_critical}}
\label{sec:near_critical}

\subsection{Preliminaries}
In the proof of Theorem~\ref{thm:near_critical}, we make use of the corresponding result for the simple random walk. 
Let $\mu_c = \abs\Omega\inv = (2d)\inv$ denote its critical point. 
For $\mu \in (0, \mu_c]$, we define the \emph{mass} $m_0(\mu) \ge 0$ to be the unique solution to 
\begin{equation} \label{eq:m0_def}
\cosh m_0(\mu) = 1 + \frac{ 1 - \mu\abs\Omega }{ 2\mu }, 
\end{equation}
so that $m_0(\mu_c) = 0$ and the function $\mu \mapsto m_0(\mu)$ is strictly decreasing. 
It is shown in \cite[Theorem~A.2]{MS93} that $m_0(\mu)$ is the exponential decay rate of $C_\mu(x)$ and $C_\mu(x) \le C_\mu(0) e^{-m_0(\mu) \norm x_\infty}$. 
By inverting and expanding the $\cosh$, we have
\begin{equation}
\label{eq:m0}
m_0(\mu)^2 = \frac{ 1 - \mu \abs\Omega }{ \mu } + O \bigg(\frac{ 1-\mu\abs\Omega}{2\mu} \bigg)^2 
= \abs\Omega ( 1 - \mu \abs\Omega ) + O (1-\mu\abs\Omega)^2 
\end{equation}
as $\mu \to \mu_c = \abs\Omega\inv$.
The following lemma, proved in \cite[Proposition~2.1]{Slad23_wsaw} using heat kernel estimates,
establishes Theorem~\ref{thm:near_critical} for $C_\mu(x)$.

\begin{lemma}
\label{lem:C}
For $d>2$, there are contants $a_0 > 0$ and $a_1 \in (0,1)$ such that for all $\mu \in (0, \mu_c]$ and $x\in \Z^d$,
\begin{equation}\label{eq:Cmu_bd}
C_\mu(x) \le \frac{ a_0 }{ \nnnorm x^{d-2} } e^{-a_1 m_0(\mu) \norm x _ \infty}.
\end{equation}
\end{lemma}

To prove Theorem~\ref{thm:near_critical}, we compare $G_z$ to a random walk $C_\mu$ with a good parameter $\mu = \mu_z$, chosen so that the error of such a comparison is small when $x$ is large. 
This isolation of leading term idea is also used in \cite{Slad23_wsaw, LS24a}.

Let $F_z$ satisfy Assumption~\ref{ass:Fz}. 
For $\mu \in (0,\mu_c]$, we define $A_\mu = \delta - \mu\abs\Omega D$,
so that $A_\mu * C_\mu = \delta$ by \eqref{eq:LGF}. 
For any $\lambda\in \R$, we write
\begin{align}
G_z
&= \lambda C_\mu + \delta * G_z - \lambda C_\mu * \delta	\nl
&= \lambda C_\mu + (C_\mu * A_\mu) * G_z - \lambda C_\mu * (F_z * G_z)\nl
&= \lambda C_\mu + C_\mu * E_{z,\lambda,\mu} * G_z,
\end{align}
with $E_{z,\lambda,\mu} = A_\mu - \lambda F_z$.
We choose $\lambda$ and $\mu$ to ensure
\begin{equation} \label{eq:E_condition}
\sum_{x\in \Z^d} E_{z,\lambda,\mu}(x) = \sum_{x\in \Z^d} \abs x^2 E_{z,\lambda,\mu}(x) = 0.	
\end{equation}
This is a system of two linear equations in $\lambda,\mu$, with solution 
\begin{equation} \label{eq:lambda_z}
\lambda_z = \frac{1} { \hat F_z(0) - \sum_{x\in \Z^d} \abs x^2 F_z(x) },
\qquad \mu_z\abs\Omega = 1 - \lambda_z \hat F_z(0).
\end{equation}
It follows from a quick computation that $\lambda_z >0$ and $\muz\in (0, \mu_c]$ (see Lemma~\ref{lem:mass}), so $C_\muz$ is well-defined. 
With this choice of $\lambda_z, \mu_z$, we decompose $G_z$ as
\begin{equation}  \label{eq:G_isolate}
G_z = \lambda_z C_\muz + f_z,
\qquad
f_z =  C_\muz * E_z * G_z,
\end{equation}
where $E_z = E_{z,\lambda_z, \muz}$. 
Multiplying by $e^{mx_1}$ then gives the tilted equation
\begin{equation} \label{eq:Gm_decomp}
G_z\supm = \lambda_z C_\muz\supm + f_z\supm,
\qquad
f_z\supm =  C_\muz\supm * E_z\supm * G_z\supm. 
\end{equation}

The next proposition provides the key estimate on the remainder $f_z\supm$ in the Fourier space. 
For a function $g : \T^d \to \C$
and $p \in [1,\infty)$ we denote $L^p(\T^d)$ norms with a subscript $p$:
\begin{equation}
    \norm g _p = \biggl( \int_{\T^d} |g(k)|^p \frac{dk}{(2\pi)^d} \bigg)^{1/p},
\end{equation}
and we write $\norm g_\infty$ for the supremum norm.
Given a multi-index $\alpha = (\alpha_1, \dots, \alpha_d) \in \N_0^d$ and $x\in \R^d$, 
we write $\abs \alpha = \sum_{i=1}^d \alpha_j$, $x^\alpha = \prod_{i=1}^d x_i^{\alpha_i}$, and $\grad^\alpha = \prod_{i=1}^d \grad_i^{\alpha_i}$.

\begin{proposition} \label{prop:f}
Under Assumption~\ref{ass:Fz}, 
there are constants $\delta_0 >0$ and $\sigma_0 \in (0,1)$
such that, for all $z \in [ z_c - \delta_0, z_c )$ and $m \in [0, \sigma_0 \mz]$, 
the function $\hat f_z\supm = \hat C_\muz\supm \hat E_z\supm \hat G_z\supm$ is $d-2$ times weakly differentiable.
Moreover, we have 
\begin{equation} \label{eq:f_alpha}
\bignorm{ \grad^\alpha \hat f_z\supm }_1 \lesssim 1
	\qquad (\abs \alpha \le d-2)
\end{equation}
with the constant independent of $z,m$.
\end{proposition}

When combined with the following general lemma, which relates the smoothness of a function on $\T^d$ to the decay of its Fourier coefficients, Proposition~\ref{prop:f} gives a good bound on $f_z\supm(x)$. 
The notion of \emph{weak derivative} is useful in relaxing moment assumptions on $F_z\supm(x)$; 
indeed, if the function $\abs x^{d-2} F_z\supm(x)$ in Assumption~\ref{ass:Fz}(v) is in $\ell^1(\Z^d)$, then $\hat f_z\supm$ will be $d-2$ times classically differentiable. 
Weak derivatives are defined by the usual integration by parts formula, and properties of the weak derivative that we use can be found in \cite[Appendix~A]{LS24a}.

\begin{lemma}
\label{lem:fourier}
Let $a,d>0$ be positive integers.
Suppose that $\hat h : \T^d \to \C$ is  $a$ times weakly differentiable,
and let $h : \Z^d \to \C$
be the inverse Fourier transform of $\hat h$.
There is a constant $c_{d,a}$ depending only on the dimension $d$ and the
order $a$ of differentiation, such that
\begin{equation}
\label{eq:Graf}
    |h(x)|
    \le
    c_{d,a} \frac{1}{\xvee^a}
    \max_{|\alpha| \in \{0,a\}}\|\grad^\alpha \hat h\|_1 .
\end{equation}
\end{lemma}

\begin{proof}
The bound \eqref{eq:Graf} follows from
\cite[Corollary~3.3.10]{Graf14}.
Although stated for classical derivatives, the proof of \cite[Corollary~3.3.10]{Graf14}
applies for weak derivatives since it only uses integration by parts.
\end{proof}

We first prove Theorem~\ref{thm:near_critical} assuming Proposition~\ref{prop:f}.
We also need the following conclusion of Lemma~\ref{lem:mass}:
under Assumption~\ref{ass:Fz} there are constants $\delta_1 >0$ and $\sigma_1 \in (0,1)$ such that for all $z \in [z_c - \delta_1, z_c)$,
\begin{equation} \label{eq:mass_claim}
\lambda_z = O(1), \qquad
\sigma_1 m(z) \le a_1 m_0(\mu_z),
\end{equation}
where $a_1$ is the constant in \eqref{eq:Cmu_bd}.

\begin{proof}[Proof of Theorem~\ref{thm:near_critical}]
We take $\delta = \delta_0 \wedge \delta_1$ and 
$\sigma = \sigma_0 \wedge \sigma_1$,
so that the conclusion of Proposition~\ref{prop:f} holds for all $z \in [z_c - \delta, z_c)$ and all $ m \in [0, \sigma m(z)]$. 
Using Lemma~\ref{lem:fourier} with $a = d-2$, we have $\abs{ f_z\supm(x) } \lesssim \nnnorm x^{-(d-2)}$ with a constant independent of $z,m$. 
From the tilted equation \eqref{eq:Gm_decomp}, we then get
\begin{equation} \label{eq:Gm_decomp2}
G_z\supm(x) = \lambda_z C_\muz\supm(x) + O\bigg( \frac 1 { \nnnorm x^{d-2} } \bigg). 
\end{equation}
Since $m \le \sigma_1 \mz \le a_1 m_0(\muz)$, by Lemma~\ref{lem:C} we have
\begin{equation}
C_\muz\supm(x) = C_\muz(x) e^{mx_1}
\le \frac{ a_0 }{ \nnnorm x^{d-2} } e^{ (-a_1 m_0(\muz) + m) \norm x _ \infty} 
\le \frac{ a_0 }{ \nnnorm x^{d-2} } e^0. 
\end{equation}
Inserting the above and $\lambda_z = O(1)$ into \eqref{eq:Gm_decomp2},
we obtain
\begin{equation} \label{eq:Gm_bd}
G_z(x) e^{mx_1}  = G_z\supm(x) 
\le  O(1) \frac{ a_0 }{ \nnnorm x^{d-2} } + O\bigg( \frac 1 { \nnnorm x^{d-2} } \bigg)
\lesssim \frac 1 {\nnnorm x^{d-2}}.
\end{equation}
By the symmetry of $G_z$, we can assume $x_1 = \norm x_\infty$.
Then, by taking $m=\sigma \mz$ in \eqref{eq:Gm_bd} and by using $\norm x_\infty \ge d^{-1/2} \abs x$, we get
\begin{align}
G_z(x) \lesssim \frac{ 1 }{ \nnnorm x^{d-2} } e^{- \sigma\mz \norm x_\infty}
\le \frac{ 1 }{ \nnnorm x^{d-2} } e^{- (\sigma d^{-1/2} ) \mz \abs x},
\end{align}
which is the desired result with $c_1 = \sigma d^{-1/2}  > 0$. 
\end{proof}

To complete the proof of Theorem~\ref{thm:near_critical},
it remains to prove Proposition~\ref{prop:f} and Lemma~\ref{lem:mass}.

\subsection{Proof of Proposition~\ref{prop:f}}

Proposition~\ref{prop:f} is about derivatives of 
\begin{equation}
\hat f_z \supm =  \hat C_\muz\supm \hat E_z\supm \hat G_z\supm
= \frac{ \hat E_z\supm } { \hat A_\muz \supm \hat F_z \supm }.
\end{equation}
To show the above is $d-2$ times weakly differentiable, we use the product and quotient rules of weak differentiation given in \cite[Lemmas~A.2--A.3]{LS24a}. 
The rules state that if a naive application of the usual calculus rules to compute $\grad^\alpha \hat f_z \supm$ produces integrable results for all $\abs\alpha \le d-2$, then $\hat f_z\supm$ is $d-2$ times weakly differentiable, with derivatives given by the usual calculus rules.
We therefore compute all terms of the form
(omitting subscripts $z, \muz$ and superscript $(m)$)
\begin{equation} \label{eq:f_decomp}
\frac{ \prod_{n=1}^i \grad^{\delta_n} \hat A  }{ (\hat A) ^{1+i} }
\big( \grad^{\alpha_2} \hat E \big)
\frac{ \prod_{l=1}^j \grad^{\gamma_l} \hat F  }{ ( \hat F) ^{1+j} }
=
\( \prod_{n=1}^i \frac{ \grad^{\delta_n} \hat A }{\hat A } \)
\( \frac{ \grad^{\alpha_2} \hat E }{ \hat A \hat F } \)
\( \prod_{l=1}^j \frac{ \grad^{\gamma_l} \hat F  }{ \hat F } \),
\end{equation}
where $\alpha = \alpha_1 + \alpha_2 + \alpha_3$, $ 0 \le i \le \abs {\alpha_1}$, $0 \le j \le \abs {\alpha_3 }$, $\sum_{n=1}^i \delta_n = \alpha_1$, and $\sum_{l=1}^j \gamma_l = \alpha_3$.
A linear combination of these terms will give $\grad^\alpha \hat f_z \supm$. 
We prove $\eqref{eq:f_decomp} \in L^1(\T^d)$ using H\"older's inequality and the following lemma. 

\begin{lemma} \label{lem:AFEbds}
Under Assumption~\ref{ass:Fz}, 
there are constants $\delta_0 >0$, $\sigma_0 \in (0,1)$, and $\eps' \in (0, \eps \wedge 1]$
such that, for all $z \in [ z_c - \delta_0, z_c )$ and $m \in [0, \sigma_0 \mz]$,
the functions $\hat F_z\supm, \hat A_\muz\supm, \hat E_z\supm$ are $d-2$ times weakly differentiable.
Moreover, for any $\abs \gamma \le d-2$ and
any choice of $q_1, q_2$ satisfying 
\begin{equation}
    q_1\inv > \frac{ \abs \gamma } d,
    \qquad
    q_2\inv > \frac{ 2 - \eps' + \abs \gamma } d, 
\end{equation}
we have 
\begin{equation}
\label{eq:AFE}
\biggnorm{ \frac{ \grad^\gamma \hat A\supm_\muz }{ \hat A\supm_\muz } }_{q_1}, \;
\biggnorm{ \frac{ \grad^\gamma \hat F_z\supm }{ \hat F_z\supm } }_{q_1}, \;
\biggnorm{ \frac{ \grad^\gamma \hat E_z\supm }{ \hat A_\muz\supm \hat F_z\supm }}_{q_2}
\lesssim 1
\end{equation} 
with the constants independent of $z, m$. 
\end{lemma}

\begin{proof}[Proof of Proposition~\ref{prop:f} assuming Lemma~\ref{lem:AFEbds}]
We need to show all terms of the form \eqref{eq:f_decomp} are integrable.
By Lemma~\ref{lem:AFEbds} and 
H\"older's inequality, the $L^r(\T^d)$ norm of \eqref{eq:f_decomp} is bounded by
\begin{equation} \label{eq:Holder_pf}
\( \prod_{n=1}^i \biggnorm{ \frac{ \grad^{\delta_n} \hat A }{\hat A }}_{r_n} \)
\biggnorm{ \frac{ \grad^{\alpha_2} \hat E }{ \hat A \hat F } }_q
\( \prod_{l=1}^j \biggnorm{ \frac{ \grad^{\gamma_l} \hat F  }{ \hat F }}_{q_l} \),
\end{equation}
where $r$ can be any positive number that satisfies
\begin{align}
\frac 1 r &= \( \sum_{n=1}^i \frac 1 {r_n} \) + \frac 1 q + \( \sum_{l=1}^j \frac 1 {q_l} \) \nl
&> \frac{ \sum_{n=1}^i \abs{\delta_n}  } d  + \frac{ 2 - \eps' + \abs{\alpha_2}   } d + \frac{ \sum_{l=1}^j \abs{\gamma_l}  } d
= \frac{  \abs \alpha  + 2 - \eps' } d.
\end{align}
Since $\abs \alpha \le d-2$ and $\eps' > 0$,
the allowed values of $r$ includes $r = 1$.
This shows that $\eqref{eq:f_decomp}\in L^1(\T^d)$ and proves that $\hat f_z\supm$ is $d-2$ times weakly differentiable. 
Moreover, 
since $\norm{ \grad^\alpha \hat f_z\supm}_1$ is bounded by a linear combination of terms like \eqref{eq:Holder_pf}, and since all factors of \eqref{eq:Holder_pf} are bounded by constants independent of $z,m$ by \eqref{eq:AFE}, we get a uniform bound on $\norm{\grad^\alpha  \hat f_z\supm}_1$, as desired.
\end{proof}

In the proof of Lemma~\ref{lem:AFEbds}, we will use the following lemma to estimate derivatives of order larger than~2. The proof uses boundedness of the Fourier transform and H\"older's inequality on $\Z^d$. 

\begin{lemma}\label{lem:F}
Under Assumption~\ref{ass:Fz}(i) and (v), the function $\hat F_z\supm$ is $(d-2)\vee 2$ times weakly differentiable, and there exists $\eps'' \in (0,1]$ such that
\begin{alignat}{2} \label{eq:F_gamma_2}
&\bignorm{ \grad^\gamma \hat F_z\supm }_{ \infty } 
	\lesssim 1  	\qquad &&( \abs \gamma = 2 ), \\
&\bignorm{ \grad^\gamma \hat F_z\supm }_{ \frac d {\abs \gamma - 2 - \eps''} }
	\lesssim 1 	\qquad &&(3 \le \abs \gamma \le d-2 ), 
	\label{eq:F_gamma_3}
\end{alignat}
with the constants independent of $z, m$. 
\end{lemma}

\begin{proof}
For simplicity, we write $F = F_z\supm$. All estimates are quantitative and are uniform in $z,m$. 
The claim \eqref{eq:F_gamma_2} on $\abs \gamma=2$ follows directly from Assumption~\ref{ass:Fz}(i) and boundedness of the $L^1$ Fourier transform. 
For $3 \le \abs \gamma \le d-2$, which only happens when $d\ge 5$, 
we use boundedness of the $L^p$ Fourier transform ($1\le p \le 2$). 
Let $a = d-2$, $b = \abs \gamma$. We decompose
\begin{equation}
\abs x^b F(x) =  \big( \abs x^{ 2 } F(x) \big)^{ \frac{a-b}{a-2} } 
	\big( \abs x^{ a } F(x) \big)^{ \frac{b-2}{a-2} }.
\end{equation}
Since $\abs x^2 F(x) \in \ell^1(\Z^d)$ by Assumption~\ref{ass:Fz}(i) and $\abs x^a F(x) \in \ell^{p \wedge 2 }(\Z^d)$ for some $p < d/4$ by Assumption~\ref{ass:Fz}(v), 
it follows from H\"older's inequality that
\begin{equation}
\abs x^b F(x) \in \ell^{p_b \wedge 2}(\Z^d)
\quad \text{with} \quad
\frac 1 {p_b} = \frac{ a-b }{a-2} + \frac { b-2 }{ (a-2)p }. 
\end{equation}
Taking the Fourier transform, we get that the weak derivative $\grad^\gamma \hat F$ exists \cite[Lemma~A.4]{LS24a}, and we get that
\begin{equation}
\grad^\gamma \hat F \in L^{q_b}(\T^d)
\quad \text{with} \quad
\frac 1 {q_b} = 1 - \frac 1 {p_b}
	= \frac { b-2 }{ a-2 } \( 1 - \frac 1 p \)
	< \frac{ b-2}d,
\end{equation}
where the last inequality follows from $p\inv > 4 / d = 1 - (a-2)/d$. 
The desired estimate \eqref{eq:F_gamma_3} then follows by picking $\eps''>0$ small enough so that $\frac{ b - 2 - \eps'' }d > q_b\inv$ for all integers $b \in [3, d-2]$. 
\end{proof}

\subsubsection*{Massive infrared bounds}
The proof of Lemma~\ref{lem:AFEbds} also uses massive infrared bounds for $\hat A_\muz\supm$ and $\hat F_z\supm$. 
We prove the bound for $\hat F_z\supm$ first. 
The next lemma is related to \cite[Lemma~5.2]{Slad23_wsaw}. 

\begin{lemma}[Massive infrared bound for $F_z\supm$]
\label{lem:massive_infrared}
Fix $0 < \sigma < 1$. 
Under Assumption~\ref{ass:Fz}(i)--(iv),
there exists $\delta_2 >0$ such that for all $z \in [ z_c - \delta_2, z_c )$ and $m \in [0,\sigma \mz]$, 
\begin{equation}
\label{eq:Fm_lb}
\abs{ \hat F_z \supm (k) } 
\gtrsim (\abs k + \mz )^2
\gtrsim (\abs k  + m)^2,
\end{equation}
with the constant independent of $z,m$ but depends on $\sigma$. 
\end{lemma}

\begin{proof}
Without loss of generality, we assume the $\eps$ in Assumption~\ref{ass:Fz}(i) satisfies $\eps \le 2$. 
We use that
\begin{equation}
\label{eq:Fm_decomp}
\abs{ \hat F_z\supm (k) }
\ge \Re \hat F_z\supm (k)
= \hat F_z\supm(0) + \Re [ \hat F_z\supm(k) - \hat F_z\supm (0) ]. 
\end{equation}
We will prove the first term to be $\gtrsim \mz^2$ and the second term to be $\gtrsim \abs k^2$ when $\delta_2$ is sufficiently small. This establishes \eqref{eq:Fm_lb}, since $a^2 + b^2 \ge \half (a + b)^2$. 

We estimate the second term of \eqref{eq:Fm_decomp} first. 
Since $F_z$ is real and symmetric, 
\begin{equation}
\Re [ \hat F_z\supm(k) - \hat F_z\supm (0) ]
=  \sum_{x\in \Z^d} ( \cos (k\cdot x) - 1) F_z(x) \cosh mx_1.
\end{equation}
Adding and subtracting 
$\hat F_z(k) - \hat F_z(0) =  \sum_{x\in \Z^d} ( \cos (k\cdot x) - 1 ) F_z(x)$
then gives
\begin{equation}
\label{eq:ReFm}
\Re [ \hat F_z\supm(k) - \hat F_z\supm (0) ]
= \hat F_z(k) - \hat F_z(0) + \sum_{x\in \Z^d} ( \cos (k\cdot x) - 1 ) F_z(x) (\cosh (mx_1) - 1). 
\end{equation}
By the elementary inequalities (here we use $\eps\le 2$)
\begin{equation}
\label{eq:cosh1}
\abs{  \cos (k\cdot x) - 1 } \lesssim \abs{ k\cdot x }^2 \le \abs k^2 \abs x^2,
\qquad
\abs{ \cosh (mx_1) - 1 } \lesssim \abs{ mx_1}^\eps \cosh mx_1, 
\end{equation}
the last term of \eqref{eq:ReFm} is bounded in absolute value by a constant multiple of
\begin{equation}
m^\eps \abs k^2 \sum_{x\in \Z^d} \abs x^2 \abs{ x_1}^\eps \abs {F_z(x) } \cosh mx_1
\le m^\eps \abs k^2 \sum_{x\in \Z^d} \abs x^{2+\eps} \abs {F_z\supm(x) }
\le m^\eps \abs k^2 K_1.
\end{equation}
Therefore, using the untilted infrared bound (Assumption~\ref{ass:Fz}(iii)) in the first term of \eqref{eq:ReFm}, we get
\begin{equation}
\Re [ \hat F_z\supm(k) - \hat F_z\supm (0) ]
\ge (K_2 - O(m^\eps) ) \abs k^2.
\end{equation}
Since $m < \mz$ and $\mz \to 0$ as $z\to z_c$, taking $\delta_2$ sufficiently small ensures $K_2 - O(m^\eps) \ge \half K_2 > 0$. 

For the first term of \eqref{eq:Fm_decomp}, take some $\sigma' \in (\sigma, 1)$ and denote $M = \sigma' m(z)$. Since $\hat F_z\supM(0) \ge 0$ by Assumption~\ref{ass:Fz}(ii), 
\begin{equation}
\label{eq:Fm0_lb}
\hat F_z\supm(0)
\ge \hat F_z\supm(0) - \hat F_z\supM(0)
= - \sum_{x\in \Z^d} F_z(x) (\cosh Mx_1 - \cosh mx_1). 
\end{equation}
Using the elementary inequality (here we use $\eps\le 2$)
\begin{equation}
\label{eq:cosh3}
\Bigabs{ \cosh t - (1 + \half t^2) } \lesssim \abs t^{2+\eps} \cosh t
\end{equation}
and $m \le \sigma \mz \le M$, we have
\begin{equation}
\cosh Mx_1 - \cosh mx_1
= \half ( M^2 - m^2 ) x_1^2 + O(M^{2+\eps}) \abs{x_1}^{2+\eps} \cosh Mx_1,
\end{equation}
so \eqref{eq:Fm0_lb} gives 
\begin{align}
\hat F_z\supm(0)
&\ge - \half (M^2 - m^2) \sum_{x\in \Z^d} x_1^2 F_z(x) 
	- O(M^{2+\eps})  \sum_{x\in \Z^d} \abs {x_1}^{2+\eps} \abs{ F_z(x) } \cosh Mx_1 \nl
&= \half (M^2 - m^2) \del_1\del_1 \hat F_z(0)
	- O(M^{2+\eps})  \sum_{x\in \Z^d} \abs {x_1}^{2+\eps} \abs{ F_z\supM(x) }. 
\end{align}
Since $M  < m(z)$, the second term is $O(M^{2+\eps})$ by hypothesis. 
Also, by Taylor's Theorem and the untilted infrared bound, $\half \del_1\del_1 \hat F_z(0) \ge K_2$, hence
\begin{equation}
\label{eq:Fm0_lb2}
\hat F_z\supm(0)
\ge K_2 (M^2 - m^2) - O(M^{2+\eps})
\ge K_2 (\sigma'^2 - \sigma^2) \mz^2 - O(\mz^{2+\eps})
\gtrsim m(z)^2
\end{equation}
for $m(z)$ sufficiently small, \ie, when $\delta_2$ is sufficiently small. 

This completes the proof of Lemma~\ref{lem:massive_infrared}.
\end{proof}

A massive infrared bound for $A_\mu\supm$ follows as a corollary of Lemma~\ref{lem:massive_infrared}. The result in fact holds with $\delta_A = (2\abs\Omega)\inv$, see \cite[Lemma~2.2]{Slad23_wsaw} which works with $\hat A_\mu\supm$ directly. 

\begin{corollary} [Massive infrared bound for $A_\mu\supm$]
\label{cor:massive_infrared}
Fix $0 < \sigma' < 1$. 
Recall $\mu_c = \abs\Omega\inv = (2d)\inv$ and the mass $m_0(\mu)$ of $C_\mu$ from \eqref{eq:m0_def}. 
There exists $\delta_A >0$ such that for all $\mu \in [ \mu_c - \delta_A, \mu_c )$ and $m \in [0,\sigma' m_0(\mu)]$, 
\begin{equation} \label{eq:Am_lb}
\abs{ \hat A_\mu \supm (k) } 
\gtrsim (\abs k + m_0(\mu) )^2
\gtrsim (\abs k  + m)^2,
\end{equation}
with the constant independent of $\mu,m$ but depends on $\sigma'$. 
\end{corollary}

\begin{proof}
We use Lemma~\ref{lem:massive_infrared} with $z = \mu$, $F_z = A_\mu = \delta - \mu\abs\Omega D$, and $\mz = m_0(\mu)$. 
To verify Assumption~\ref{ass:Fz}(i)--(iv), 
(i) follows from $D(x)$ having finite support,
(iv) follows directly from the asymptotics of $m_0(\mu)$ in \eqref{eq:m0},
and (ii)--(iii) follows from direct calculations.
For (ii), since $D(x) = \frac 1 {2d}  \1 \{\abs x = 1\}$, 
\begin{align}
\hat A_\mu\supm(0) = 1 - \mu\abs\Omega \hat D\supm(0)
&= 1 - \frac {\mu\abs\Omega} d ( d - 1 + \cosh m ) 		 \nl
&\ge 1 - \frac {\mu\abs\Omega} d ( d-1 + \cosh m_0(\mu) ) = 0, 
\end{align}
and $\hat A_\mu(0) = 1 - \mu \abs\Omega \to 0$ as $\mu \to \abs\Omega\inv =  \mu_c$. 
For (iii), again using the formula of $D(x)$, 
\begin{equation}
\hat A_\mu(k) - \hat A_\mu(0)
= \mu\abs\Omega [ \hat D(0) -  \hat D(k) ]
= \frac {\mu\abs\Omega} d \sum_{j=1}^d (1 - \cos k_j) \ge \frac {4\mu} {\pi^2}\abs k^2,
\end{equation}
so we can take $K_2 = 2\mu_c/ \pi^2$ as long as $\delta_A \le \half \mu_c$. 
This completes the proof.
\end{proof}

\begin{lemma}
\label{lem:mass}
Under Assumption~\ref{ass:Fz}(i)--(iv), 
there exists $\delta_3 >0$ such that for all $z \in [ z_c - \delta_3, z_c )$,
\begin{enumerate}[label=(\roman*)]
\item
$\lambda_z$ is positive, bounded, and bounded away from 0;

\item
$\mu_z\abs\Omega \le 1$, and $\mu_z\abs\Omega \to 1$ as $z\to z_c$;

\item
$ m(z) \lesssim m_0(\mu_z) $.
\end{enumerate}
\end{lemma}

\begin{proof}
Recall from \eqref{eq:lambda_z} that
\begin{equation}
\lambda_z = \frac{1} { \hat F_z(0) - \sum_{x\in \Z^d} \abs x^2 F_z(x) }
	= \frac 1 { \hat F_z(0) + \kappa_z},
\qquad \mu_z\abs\Omega = 1 - \lambda_z \hat F_z(0), 
\end{equation}
where the second equality defines $\kappa_z$.

\smallskip\noindent
(i) Since $\hat F_z (0)$ is non-negative and converges to 0 as $z\to z_c$ by Assumption~\ref{ass:Fz}(ii), it suffices to show $\kappa_z = - \sum_{x\in \Z^d} \abs x^2 F_z(x)$ is positive, bounded, and bounded away from $0$. 
The boundedness follows directly from $\sum_{x\in \Z^d} \abs x^{2+\eps} \abs{ F_z(x) } \le K_1$ uniformly.
To see $\kappa_z$ is positive and bounded away from 0, we use Taylor's Theorem and the untilted infrared bound:
\begin{equation}
 \kappa_z =  - \sum_{x\in \Z^d} \abs x^2 F_z(x)
	= \Delta \hat F_z(0) \ge 2d K_2 > 0.
\end{equation}

\noindent
(ii) This follows from part (i) and Assumption~\ref{ass:Fz}(ii).

\smallskip\noindent
(iii)
By the asymptotics of $m_0(\cdot)$ in \eqref{eq:m0}, the definition of $\mu_z$, and the $m=0$ case of \eqref{eq:Fm0_lb2} (with, say, $\sigma=\half$), we have
\begin{equation}
m_0(\mu_z)^2
\ge \frac {\abs\Omega} 2 (1 - \mu_z \abs\Omega)
= \frac{ \abs\Omega } 2 \lambda_z \hat F_z(0)
\gtrsim \frac{ \abs\Omega } 2 \lambda_z m(z)^2. 
\end{equation}
The desired result then follows from $\lambda_z$ being bounded away from 0.
\end{proof}

We now prove \eqref{eq:mass_claim}.
In fact, we can take $\delta_1 = \delta_3$.
The claim $\lambda_z = O(1)$ follows from Lemma~\ref{lem:mass}(i). 
For the other claim, we use Lemma~\ref{lem:mass}(iii) and choose $\sigma_1 > 0$ sufficiently small to make
\begin{equation}
\sigma_1 m(z) \le a_1 m_0(\muz)
\qquad (z \in [z_c - \delta_3, z_c)),
\end{equation}
as desired.
Also, 
by setting $\sigma' = a_1 \in (0,1)$ in Corollary~\ref{cor:massive_infrared},
and by using $\mu_z \to \abs\Omega\inv = \mu_c$ from Lemma~\ref{lem:mass}(ii), we obtain
\begin{equation} \label{eq:A_massive}
\abs{ \hat A_\muz\supm (k) } \gtrsim (\abs k + m)^2
	\qquad ( z \in [z_c - \delta_4, z_c),\  m \in [0, \sigma_1 m(z)] )
\end{equation}
for some $\delta_4 > 0$,
with the constant independent of $z,m$. 
By setting $\sigma = \sigma_1$ in Lemma~\ref{lem:massive_infrared}, and by redefining $\delta_4$ to be smaller if necessary, we have
\begin{equation} \label{eq:F_massive}
\abs{ \hat F_z \supm (k) } 
\gtrsim (\abs k  + m)^2
	\qquad ( z \in [z_c - \delta_4, z_c),\  m \in [0, \sigma_1 m(z)] )
\end{equation}
too.

\subsubsection*{Proof of Lemma~\ref{lem:AFEbds}}
We use the massive infrared bounds in conjugation with the following lemma, 
which translates the good moment behaviour of $E_z$ in \eqref{eq:E_condition}
into good bounds on $\hat E_z\supm$ and its derivatives. 
Cancellation of powers of $(\abs k + m)$ ultimately allows $d-2$ derivatives of $\hat f_z\supm$ to be taken. 

\begin{lemma}
\label{lem:E}
Under Assumption~\ref{ass:Fz}(i)--(iv),
there exists $\delta_5 >0$ such that for all $z \in [ z_c - \delta_5, z_c )$ and all $m \in [0, \mz)$, 
\begin{equation} \label{eq:E_gamma_ub}
\abs{ \grad^\gamma \hat E_z \supm (k) } 
\lesssim  (\abs k  + m)^{2 + (\eps \wedge 1) - \abs \gamma }
\qquad ( \abs \gamma \le 2 ),
\end{equation}
with the constant independent of $z,m$. 
\end{lemma}

\begin{proof}
Since $E_z = A_\muz - \lambda_z F_z$
and $A_\muz = \delta - \muz\abs\Omega D$ has finite support, 
it follows from Lemma~\ref{lem:mass} and the assumptions that $\sum_{x\in \Z^d} \abs x^{2+\eps} \abs{ E_z\supm(x)}$ is bounded uniformly in $z \in [ z_c - \delta_5, z_c )$ and in $m \in [0, \mz)$, for some $\delta_5>0$. 
For simplicity, we write $E = E_z$ and $\eps_1 = \eps \wedge 1$ next. 
Note that all second moments of $E(x)$ vanish: 
$\sum_{x\in \Z^d} x_i x_j E(x) = 0$ for $i\ne j$ because $E$ is even, and $\sum_{x\in \Z^d} x_i^2 E(x) = 0$ for all $i$ by \eqref{eq:E_condition} and symmetry.
The claim \eqref{eq:E_gamma_ub} essentially follows from expanding the function $(k,m)\mapsto \hat E\supm(k)$ around $(0,0)$ up to $2+\eps_1$ orders, since derivatives of $\hat E\supm (k)$ at $(0,0)$ are exactly moments of $E(x)$.

In detail, using symmetry we write
\begin{equation}
\hat E\supm(k) = 
	\sum_{x\in \Z^d} \cos (k\cdot x) \cosh (mx_1) E(x) 
	+ i \sum_{x\in \Z^d} \sin (k\cdot x) \sinh (mx_1) E(x). 
\end{equation}
Since $\hat E(0) = \sum_{x\in \Z^d} E(x) = 0$, we can decompose $\hat E\supm(k)$ as 
\begin{equation}
\label{eq:Em_decomp}
\hat E\supm(k)
	= \hat E\supm(0) + \hat E(k) + \varphi_m(k)
	\quad \text{with}\quad
	\varphi_m(k) = \hat E\supm(k) - \hat E\supm(0) - \hat E(k) + \hat E(0) .
\end{equation}
We will show all three terms of \eqref{eq:Em_decomp} satisfy the desired bound \eqref{eq:E_gamma_ub}. 

For $\hat E\supm(0)$, since $\sum_{x\in \Z^d} E(x) = \sum_{x\in \Z^d} x_1^2 E(x) = 0$, we have
\begin{equation}
\hat E\supm(0) = \sum_{x\in \Z^d} \cosh(mx_1) E(x)
	= \sum_{x\in \Z^d} \Big( \cosh(mx_1) - 1 - \frac{ m^2 x_1^2 } 2 \Big) E(x),
\end{equation}
and the desired bound follows from \eqref{eq:cosh3} with $\eps_1$:
\begin{equation}
\abs{ \hat E\supm(0) } \lesssim m^{2+\eps_1} \sum_{x\in \Z^d}  \abs {x_1}^{2+\eps_1} \abs{ E\supm(x) }
	\lesssim (\abs k + m)^{2+\eps_1}. 
\end{equation}
For $\hat E(k)$, we define $g_x(k) = \cos(k\cdot x) - 1 + \frac{ (k\cdot x)^2 }{2!}$.
Using \eqref{eq:E_condition} and symmetry, 
\begin{equation}
\hat E(k) = \sum_{x\in \Z^d} E(x) \cos(k\cdot x) = \sum_{x\in \Z^d} E(x) g_x(k).
\end{equation}
Since $\eps_1 \le 2$, by elementary properties of sine and cosine, we have
$\abs{ \grad^\gamma g_x(k) } 
\lesssim \abs k^{2+\eps_1 - \abs \gamma} \abs x^{2+\eps_1}$, so
\begin{equation}
\abs{ \grad^\gamma \hat E (k) }
\le \sum_{x\in \Z^d} \abs {E(x)} \abs{ \grad^\gamma g_x(k) }
\lesssim \abs k^{2+\eps_1 - \abs \gamma} \sum_{x\in \Z^d} \abs x^{2+\eps_1} \abs{ E(x) }
\lesssim ( \abs k + m )^{2 + \eps_1 - \abs \gamma}. 
\end{equation}

For $\varphi_m(k)$, as in \eqref{eq:ReFm}, we write 
\begin{equation}
\varphi_m(k) = \sum_{x\in \Z^d} (\cos(k\cdot x) - 1 ) (\cosh(mx_1) - 1 ) E(x) 
	+ i \sum_{x\in \Z^d} \sin(k\cdot x) \sinh(mx_1) E(x). 
\end{equation}
To bound the real part and its derivatives, we use
$\abs{ \grad^\gamma ( \cos (k\cdot x) - 1  )} \lesssim \abs k^{2-\abs \gamma} \abs x^2$ 
and the second inequality of \eqref{eq:cosh1}, which give
\begin{equation}
\bigabs{  \Re \grad^\gamma \varphi_m(k) }
\lesssim \abs k^{2- \abs \gamma} m^{\eps_1} \sum_{x\in \Z^d} \abs x^2 \abs{x_1}^{\eps_1} \abs{ E\supm(x) }
\lesssim (\abs k + m)^{2 + \eps_1 - \abs \gamma}. 
\end{equation}
To bound the imaginary part and its derivatives, we write
\begin{align}
\Im \varphi_m(k)
&= \sum_{x\in \Z^d} \( \sin(k\cdot x) - k\cdot x \) \sinh(mx_1) E(x) 
	+ \sum_{x\in \Z^d}  (k\cdot x) \sinh(mx_1) E(x) \nl
&= \sum_{x\in \Z^d} \( \sin(k\cdot x) - k\cdot x \) \sinh(mx_1) E(x) 
	+ \sum_{x\in \Z^d}  (k\cdot x) \( \sinh(mx_1) - mx_1 \) E(x) ,
\end{align}
where the second equality follows from $\sum_{x\in \Z^d} x_i x_j E(x) = 0$ for all $i,j$. 
Using the elementary inequalities (here we use $\eps_1 \le 1$)
\begin{align}
\abs{ \grad^\gamma (  \sin(k\cdot x) - k\cdot x  ) } &\lesssim \abs k^{2+\eps_1 - \abs \gamma} \abs x^{2 + \eps_1},  	\nl
\abs{ \sinh mx_1 } &\le \cosh mx_1, 	\\ \nonumber
\abs{ \sinh(mx_1) - mx_1 } &\lesssim \abs{ mx_1 }^{1+\eps_1} \cosh mx_1,
\end{align}
we get 
\begin{equation}
\begin{aligned}
\bigabs{  \Im \grad^\gamma \varphi_m(k) }
&\lesssim \( \abs k^{2 + \eps_1- \abs \gamma}     
	+ \abs k^{1-\abs \gamma} m^{1+\eps_1} \1_{\abs \gamma \le 1} \) 
	\sum_{x\in \Z^d} \abs x^{2+\eps_1} \abs{ E\supm(x) }	 \\
&\lesssim (\abs k + m)^{2 + \eps_1 - \abs \gamma}. 
\end{aligned}
\end{equation}
This completes the proof of Lemma~\ref{lem:E}.
\end{proof}

\begin{proof}[Proof of Lemma~\ref{lem:AFEbds}]
We take $\delta_0 = \delta_3 \wedge \delta_4 \wedge \delta_5$ and $\sigma_0 = \sigma_1$,
so that all estimates above apply. 
The differentiability of $\hat F_z\supm$ has been established in Lemma~\ref{lem:F}, and $\hat A\supm_\muz = 1 - \muz \hat D \supm$ is classically differentiable to all orders. 
The differentiability of $\hat E_z \supm = \hat A_\muz \supm - \lambda_z \hat F_z \supm$ follows from linearity.

\medskip\noindent
\emph{Bound on $\grad^\gamma \hat A_\muz\supm / \hat A_\muz\supm$.}
For simplicity, we omit the subscript $\muz$ and write $A\supm = A_\muz\supm$ next. Let $\gamma$ be any multi-index.
There is nothing to prove for $\abs \gamma=0$ since the ratio is then $1$.
For $\abs \gamma = 1$, 
we first prove $\abs{ \grad^\gamma \hat A\supm(k) } \lesssim \abs k + m$. 
Using symmetry, 
\begin{equation}
\hat A\supm(k) = 
	\sum_{x\in \Z^d} \cos (k\cdot x) \cosh (mx_1) A(x) 
	+ i \sum_{x\in \Z^d} \sin (k\cdot x) \sinh (mx_1) A(x). 
\end{equation}
To estimate its derivative, we use the inequalities
\begin{equation}
\abs{ \grad_j \cos (k\cdot x) } \le \abs k \abs x^2,  \quad
\abs{ \grad_j \sin (k\cdot x) }  \le \abs x, \quad
\abs{ \sinh (mx_1) } \le m \abs{ x_1} \cosh (mx_1), 
\end{equation}
which give
\begin{equation} \label{eq:DAm}
\abs{ \grad_j \hat A\supm(k) } 
\le  (\abs k + m) \sum_{x\in \Z^d} \abs x^2 \abs{ A\supm(x) } 
\lesssim \abs k + m, 
\end{equation}
because $A = \delta - \muz D$ has finite support and $m \le \sigma_0 \mz \lesssim 1$. 
For $\abs \gamma \ge 2$, 
we use the uniform bound $\abs{ \grad^\gamma \hat A\supm (k) } \lesssim 1$, which is a consequence of $A(x)$ having (uniform) moments of any order.
By the massive infrared bound \eqref{eq:A_massive}, we then have
\begin{equation}
\biggnorm{ \frac{ \grad^\gamma \hat A \supm }{ \hat A\supm } (k) }_q
\lesssim \bignorm{ (\abs k + m)^{- (\abs \gamma \wedge 2 ) } }_q
\le  \bignorm{ \abs k^{-(\abs \gamma \wedge 2)} }_q
\lesssim 1
\end{equation}
uniformly in $z,m$, 
whenever $q\inv > (\abs \gamma \wedge 2) / d$.
This is stronger than the desired \eqref{eq:AFE}.

\medskip\noindent
\emph{Bound on $\grad^\gamma \hat F_z\supm / \hat F_z\supm$.}
Let $ \abs \gamma \le d-2$.
There is again nothing to prove for $\abs \gamma =0$.
The $\abs \gamma = 1$ case is exactly the same as above, using finiteness of $\sum_{x\in \Z^d} \abs x^2 \abs{ F_z \supm(x) }$ at \eqref{eq:DAm} and the massive infrared bound \eqref{eq:F_massive} instead of \eqref{eq:A_massive}.
For $2 \le \abs \gamma \le d-2$, 
it follows from Lemma~\ref{lem:F} and the monotonicity of $L^q(\T^d)$ norms in $q$ that
\begin{equation} \label{eq:F_gamma}
\bignorm{ \grad^\gamma \hat F_z\supm }_{\frac d {\abs \gamma -2}}\lesssim 1 
	\qquad (2 \le \abs \gamma \le d-2 ),
\end{equation}
with the constant independent of $z,m$. 
Since $ \abs{ 1/ \hat F_z\supm (k) } \lesssim (\abs k+m)^{-2} \le \abs k^{-2} \in L^p$ for all $p\inv > 2/d$ by the massive infrared bound, H\"older's inequality immediately gives $ \grad^\gamma \hat F_z\supm / \hat F_z\supm \in L^q$ for all $q\inv > (\abs \gamma - 2 + 2)/d$, with the $L^q$ norm bounded uniform in $z,m$, as desired.

\medskip\noindent
\emph{Bound on $\grad^\gamma \hat E_z\supm / \hat A_\muz\supm \hat F_z\supm$.}
Let $\abs \gamma \le d-2$. 
For $\abs \gamma \le 2$, writing $\eps_1 = \eps \wedge 1$, it follows from Lemma~\ref{lem:E}
and the massive infrared bounds \eqref{eq:A_massive}--\eqref{eq:F_massive} that 
\begin{equation}
\biggnorm{ \frac{ \grad^\gamma \hat E_z \supm }{ \hat A_\muz\supm \hat F_z\supm} (k) }_q
\lesssim \bignorm{ (\abs k + m)^{ 2 + \eps_1 - \abs \gamma - 4} }_q
\le  \bignorm{ \abs k^{-(2 - \eps_1 + \abs \gamma )} }_q
\lesssim 1
\end{equation}
uniformly in $z,m$, whenever $q\inv > (2 - \eps_1 + \abs \gamma )/ d$.
This is stronger than the desired \eqref{eq:AFE} because $\eps' \le \eps_1 = \eps \wedge 1$. 
For $3 \le \abs \gamma \le d-2$, it follows from $\hat E_z \supm = \hat A_\muz \supm - \lambda_z \hat F_z \supm$ and Lemma~\ref{lem:F} that
\begin{equation} 
\bignorm{ \grad^\gamma \hat E_z\supm }_{ \frac d {\abs \gamma - 2 - \eps''} } \lesssim 1 
	\qquad (3 \le \abs \gamma \le d-2 ),
\end{equation}
with the constant independent of $z,m$. 
The desired result then follows by picking $\eps' \le \eps''$ and using H\"older's inequality and the massive infrared bounds: $ \abs{ 1/ ( \hat A_\muz\supm (k) \hat F_z\supm (k) )} \lesssim (\abs k+m)^{-4} \le \abs k^{-4} \in L^p$ for all $p\inv > 4/d$, with the $L^p$ norm bounded uniform in $z,m$. 

This completes the proof of Lemma~\ref{lem:AFEbds} and concludes the proof of Theorem~\ref{thm:near_critical}.
\end{proof}

\section{Self-avoiding walk on $\Z^d$: proof of Theorem~\ref{thm:saw}}
\label{sec:saw}

The proof of Theorem~\ref{thm:saw} uses Theorem~\ref{thm:near_critical}, with the convolution equation \eqref{eq:FGz} provided by the lace expansion. 
In dimensions $d>4$, the lace expansion \cite{BS85, Slad06} produces the identity
\begin{equation}
\label{eq:conv_saw}
G_z = \delta + zD * G_z + \Pi_z * G_z
\qquad (0 \le z \le z_c),
\end{equation}
with $\Pi_z : \Z^d \to \R$ an explicit, $\Z^d$-symmetric function. This can be rearranged as 
\begin{equation} \label{eq:def_Fz}
F_z * G_z = \delta
\quad\text{with}\quad
F_z = \delta -zD - \Pi_z,
\end{equation}
which is of the form \eqref{eq:FGz}. 
For good reviews on the lace expansion for self-avoiding walk, see \cite{MS93, Slad06}. 
We recall here only the results we need.
Note that no additional numerical input, beyond those already obtained in \cite{HS92b}, is used in our analysis.
\begin{proposition} \label{prop:lace}
Let $d >4$.
\begin{enumerate}[label=(\roman*)]
\item
For all $z \in (0, z_c]$, $G_z$ is given by the Fourier integral representation \eqref{eq:Gint}. 

\item 
There is a constant $K_2 > 0$ such that for all $z \in (0, z_c]$, 
\begin{equation}
\hat F_z(k) - \hat F_z(0) \ge K_2 \abs k^2
\qquad (k\in \T^d).
\end{equation}

\item
There are constants $\kappa, \delta >0$ 
satisfying $\kappa(1+\kappa) <1$
such that for all $z \in [z_c - \delta, z_c )$ and $m\in [0, \mz)$, 
\begin{equation} \label{eq:bubble^m}
B\supm(z) = \sum_{x\ne 0} G_z\supm(x)^2 \le \kappa. 
\end{equation}

\end{enumerate}
\end{proposition}

\begin{proof}
(i) and (ii) The results are stated in \cite[Proposition~1.3]{Hara08}. The proof goes back to \cite{HS92a,HS92b}. 

\smallskip\noindent
(iii) By \cite[Theorem~2.5]{HS92a}, there is a positive constant $C_2$ satisfying $C_2(1+C_2) < 1$ such that $B^{(0)}(z_c) \le C_2$.
Taking $\gamma >0$ sufficiently small ensures that $\kappa = C_2 + \gamma$ satisfies $\kappa(1+\kappa) < 1$ still.
The desired bound on $B\supm(z)$ then follows from \cite[Lemma~3.10]{HS92a}. (The constant $\delta$ depends on $\gamma$.)
\end{proof}

The next proposition is a diagrammatic estimate on $\Pi_z$ that improves previous results. 
The improvement uses the asymptotics of $G\crit(x)$ in \eqref{eq:G_crit_asymp}. 

\begin{proposition}
\label{prop:diagram}
Let $d>4$, and let $\delta > 0$ be given by Proposition~\ref{prop:lace}(iii).
For any real number $a \in (0,d-2]$,
there is a constant $K_a$ such that
for all $z \in [z_c - \delta, z_c)$
and all $m \in [0, \mz)$, 
\begin{equation} \label{eq:Pi_moments}
\sum_{x\in \Z^d} \abs x^a \abs{ \Pi_z \supm(x) } \le K_a .
\end{equation}
\end{proposition}

\begin{proof}[Proof of Theorem~\ref{thm:saw} assuming Proposition~\ref{prop:diagram}]
We prove the desired bound for $z$ close to $z_c$ using Theorem~\ref{thm:near_critical}, 
and prove for the remaining $z < z_c - \delta$ using the exponential decay \eqref{eq:G_exp_bound}. 

To verify the hypotheses of Theorem~\ref{thm:near_critical},
we use Propositions~\ref{prop:lace} and \ref{prop:diagram}.
We know that $G_z$ is given by the Fourier integral \eqref{eq:Gint} by Proposition~\ref{prop:lace}(i). 
For Assumption~\ref{ass:Fz}, (iii) is exactly Proposition~\ref{prop:lace}(ii), and (iv) is given by the paragraph containing \eqref{eq:def_mass}. 
To verify (i) and (v), we use Proposition~\ref{prop:diagram} with $a=d-2$.
It follows from the definition of  $F_z$ in \eqref{eq:def_Fz} and $m < \mz$ that
\begin{equation}
\sum_{x\in \Z^d} \abs x^{d-2} \abs{ F_z\supm(x) } 
\le \sum_{x\in \Z^d} \abs x^{d-2} \( z  D\supm(x)   + \abs{ \Pi_z\supm(x) } \)
\le \frac{ z_c }d ( d-1 + \cosh \mz) + K_{d-2}, 
\end{equation}
which is bounded since $\mz \to 0$ as $z\to z_c$. 
This verifies Assumption~\ref{ass:Fz}(i) with $\eps = d - 4 > 0$ and Assumption~\ref{ass:Fz}(v) with $p=1$. 
Lastly, to verify Assumption~\ref{ass:Fz}(ii), we tilt equation \eqref{eq:def_Fz} and use the Fourier transform, \eqref{eq:G_exp_bound}, and $m < \mz$
to get
\begin{equation}
\hat F_z\supm(0) = \frac 1 { \hat G_z \supm (0) } 
= \frac 1 { \sum_{x\in \Z^d} G_z(x) e^{mx_1} } >0, 
\end{equation}
and we use the Monotone Convergence Theorem to get
\begin{equation}
\hat F_z(0) = \frac 1 { \hat G_z  (0) }  
= \frac 1 { \chi(z) } \to  \frac 1 { \chi(z_c) } = 0
	\qquad (z\to z_c).
\end{equation}
This verifies all hypotheses of Theorem~\ref{thm:near_critical} and prove the desired \eqref{eq:G_saw} for $z \in [z_c - \delta, z_c)$.

For the remaining $z < z_c - \delta$, 
by monotonicity we have $\norm{ G_z }_2  \le \norm{ G_{z_c-\delta} }_2  < \infty$ and $m(z) \ge m(z_c-\delta) > 0$, since $\delta > 0$.
With any choice of $c\in(0,1)$, we then have
\begin{equation}
e^{-(1-c) m(z)\norm x_\infty } 
\le e^{-(1-c) m(z_c - \delta)\norm x_\infty } 
\le \frac {C_{c,\delta,d}} { \nnnorm x^{d-2} }. 
\end{equation}
Plugging into \eqref{eq:G_exp_bound}, we get
\begin{equation} \label{eq:Gz_small}
G_z(x) \le \norm{ G_{z_c-\delta} }_2\, e^{-m(z) \norm x_\infty}
\lesssim \frac 1 { \nnnorm x^{d-2} }  e^{-c m(z)\norm x_\infty }
\le \frac 1 { \nnnorm x^{d-2} }  e^{-(c d^{-1/2}) m(z) \abs x }, 
\end{equation}
which is the desired bound. 
This proves Theorem~\ref{thm:saw}, subject to Proposition~\ref{prop:diagram}.
\end{proof}

\begin{proof}[Proof of Proposition~\ref{prop:diagram}]
This is a standard diagrammatic estimate, as in the paragraph containing \cite[(3.91)]{HS92a}. 
The argument in \cite{HS92a} uses a uniform bound on $\sup_{x\in \Z^d} \abs {x_1}^{2+\eps} G_z(x)$ and the uniformly small bubble diagram \eqref{eq:bubble^m} to prove \eqref{eq:Pi_moments} with $a = 2+\eps$ and $\eps <1/2$. 
But by the asymptotic formula of $G\crit(x)$ in \eqref{eq:G_crit_asymp} (which is proved after \cite{HS92a}), we now have
\begin{equation} 
\label{eq:Ga_sup}
\sup_{x\in \Z^d} \abs x^a G_z(x)
\le \sup_{x\in \Z^d} \abs x^{d-2} G\crit(x) \le K
	\qquad (a\le d-2)
\end{equation}
for some constant $K < \infty$. 
With \eqref{eq:Ga_sup}, the proof in \cite{HS92a} yields \eqref{eq:Pi_moments} for all $a \le d-2$. 

We now provide more details. 
By definition, the function $\Pi_z$ is an absolutely convergent series $\Pi_z(x) = \sum_{N=1}^\infty (-1)^N \Pi_z\supN(x)$ (the superscript $(N)$ here is \emph{not} an exponential tilt), where each $\Pi_z\supN$ is non-negative. 
We write $\Pi_z\supNm$ for the exponential tilt of $\Pi_z\supN$, and it suffices to prove
\begin{equation} \label{eq:Pi_moments_strong}
\sum_{x\in \Z^d} \abs x^a \sum_{N=1}^\infty \Pi_z\supNm (x) \le K_a . 
\end{equation}
We bound each $N$.
The $N=1$ term does not contribute, because $\Pi_z^{(1)}(x)$ is only nonzero when $x=0$ \cite[(4.15)]{Slad06} and $a > 0$.
For $N\ge 2$, 
we use the usual diagrammatic estimate of $\Pi\supN_z(x)$ \cite[Proposition~4.2]{Slad06}, and we distribute the tilt $e^{mx_1}$ along one side of the diagram, and distribute $\abs x^a$ along the other side of the diagram.
We illustrate this using the 4-loop term $\Pi_z^{(4,m)}$.

\begin{figure}[h]
\center{
\PiFour
\caption{
Diagrammatic representation for the upper bound of $\Pi_z^{(4)}(x)$. 
Solid lines represent $H_z$; slashed lines represent $G_z$. 
Internal vertices are summed over $\Z^d$. 
}
\label{figure:Pi4}
}
\end{figure}

Define $H_z(x) = G_z(x) - \delta_{0,x}$, so that $B\supm(z) = \norm{ H_z\supm(x) }_2^2$.
By \cite[Proposition~4.2]{Slad06}, 
\begin{equation}
\Pi_z^{(4)}(x) \le \sum_{ u,v \in \Z^d} 
	H_z(u)^2 G_z(v) H_z(u-v) G_z(x-u) H_z(x-v)^2 
\end{equation}
(see Figure~\ref{figure:Pi4}
where $u$ is the top-left vertex).
We distribute the exponential tilt along the top side using $e^{mx_1} = e^{mu_1} e^{m(x_1-u_1)}$. This gives
\begin{equation}
\Pi_z^{(4,m)}(x) \le \sum_{ u,v \in \Z^d} 
	H_z(u) H_z\supm(u) G_z(v) H_z(u-v) G_z\supm(x-u) H_z(x-v)^2.
\end{equation}
We then multiply by $\abs x^a$ and sum over $x$. 
We split $\abs x^a$ along the bottom side of the diagram using $\abs x^a \le 2^a ( \abs v^a + \abs{ x-v }^a )$. 
Then the term produced by $\abs v^a$ is
\begin{equation}
T :=  \sum_{ x,u,v \in \Z^d} 
	H_z(u) H_z\supm(u) \Big( \abs v^a G_z(v) \Big) H_z(u-v) G_z\supm(x-u) H_z(x-v)^2.
\end{equation}
By \cite[Lemma~4.6]{Slad06}, $T$ is bounded by the product of $\norm{ \abs v^a G_z(v) }_\infty$ with the $\ell^2$ norms of each of the other six factors. 
We bound the $\ell^\infty$ norm by $K$ using \eqref{eq:Ga_sup}, 
and we bound the remaining $\ell^2$ norms using Proposition~\ref{prop:lace}(iii). 
Since
\begin{equation} \label{eq:ell2_norms}
\norm{ H_z\supm }_2 = \sqrt{ B\supm(z)} \le \sqrt \kappa, \qquad
\norm{ G_z\supm}_2 = \sqrt{ 1 + \norm{ H_z }_2^2 } \le \sqrt{ 1 + \kappa }
\end{equation}
for all $z$ and all $m\in [0,\mz)$ by \eqref{eq:bubble^m},
the resultant bound is $T \le  K (\sqrt \kappa)^5 \sqrt{1+\kappa}$.
For the other term produced by $\abs{ x-v}^a$,
we put the $\ell^\infty$ norm on $\abs{x-v}^a H_z(x-v)$, and we get $K (\sqrt \kappa)^4 (\sqrt{1+\kappa})^2$ as the upper bound.
Together, this gives
\begin{equation}
\sum_{x\in \Z^d} \abs x^a \Pi_z^{(4,m)}(x) 
\le 2^{a+1} K (\sqrt \kappa)^4 (\sqrt{1+\kappa})^2 
\le 2^{a+1} K \Big(\sqrt{\kappa(1+\kappa)}\Big)^3 .
\end{equation}

For general $N \ge 2$, 
since distributing $\abs x^a$ creates a factor $\le N^a$, and 
since there are at most $N$ terms like $T$ above, the resultant bound is
\begin{equation}
\sum_{x\in \Z^d} \abs x^a \Pi_z^{(N,m)}(x)
\le K N^{a+1} \Big( \sqrt{ \kappa  (1+\kappa) } \Big)^{N-1} .
\end{equation}
Using $\kappa(1+\kappa)<1$, summing the above over $N$ gives the desired \eqref{eq:Pi_moments_strong}.
\end{proof}

\section{Self-avoiding walk on $\T^d_r$}
\label{sec:torus}

In this section we prove Theorem~\ref{thm:torus} and Corollary~\ref{cor:torus}, which concern self-avoiding walk on the discrete torus $\T^d_r = (\Z / r\Z)^d$. 
The proof of Theorem~\ref{thm:torus}, presented in Sections~\ref{sec:torus_ub}--\ref{sec:torus_lb}, uses the strategy developed in \cite{Slad23_wsaw} for weakly self-avoiding walk, but differs in the absence of a small parameter. 
This forces us to restrict to $\abs x \ge M$ with a sufficiently large $M$. 
The proof of Corollary~\ref{cor:torus} is presented in Section~\ref{sec:torus_cor}.

\subsection{Unfolding of a torus walk and consequences}
\label{sec:torus_ub}
The self-avoiding walk torus two-point function $G_z\supT(x)$ is defined by the analogue of \eqref{eq:def_G} with walks on the torus $\T^d_r$ rather than on $\Z^d$. 
Since the torus has only $V=r^d$ vertices, a self-avoiding walk can have at most length $V$. Hence, the sum \eqref{eq:def_G} in this case is a finite sum, and the torus susceptibility $\chi\supT(z) = \sum_{x\in \T^d_r} G_z\supT(x)$ is a polynomial in $z$. 

For $r\ge3$, walks on $\T^d_r$ 
(which we identify with $[ - \frac r 2, \frac r 2 )^d \cap \Z^d$) 
are in one-to-one correspondence with walks on $\Z^d$, via the canonical projection of $\Z^d$ to $\T^d_r$. We refer to the $\Z^d$ walk obtained through this correspondence as the \emph{unfolding} of the torus walk.\footnote{
In detail, if $\omega$ is a torus walk then its unfolding $\tilde \omega$ is defined by $\tilde \omega(0) = \omega(0)$ and $\tilde\omega(i) = \tilde \omega(i-1) + \omega(i) -  \omega(i-1) $ for all $i\ge 1$, with the sum computed on $\Z^d$ and then projected onto the torus.
}
Since the unfolding of a torus self-avoiding walk ending at $x\in \T^d_r$ must be a $\Z^d$ self-avoiding walk ending at $x+ru$ for some $u\in \Z^d$, we get the inequality
\begin{equation} \label{eq:unfold}
G_z\supT(x) \le \sum_{u \in \Z^d} G_z(x+ru). 
\end{equation}
The upper bound of Theorem~\ref{thm:torus} follows quickly from bounding the sum over $u\ne 0$ using Theorem~\ref{thm:saw}. 

\begin{proof}[Proof of upper bound of Theorem~\ref{thm:torus}]
By \eqref{eq:unfold}, Theorem~\ref{thm:saw}, and the inequality $\nnnorm y \ge \abs y \ge \norm y_\infty$, we have
\begin{equation}
G_z\supT(x) - G_z(x)
\le \sum_{u\ne 0} \frac{ c_0 }{ \norm{x+ru}_\infty^{d-2} } e^{-c_1\mz \norm{x+ru}_\infty }.
\end{equation}
This sum is estimated in \cite[Lemma~4.2]{HMS23}. 
Taking $a=2$ and $\nu = c_1\mz$ in the lemma, we obtain 
\begin{equation}
G_z\supT(x) - G_z(x)
\lesssim \frac{c_0 }{ c_1^2} \frac 1 { \mz^2 r^d  } e^{-\frac {c_1} 4 \mz r },
\end{equation}
so the upper bound of \eqref{eq:G^T} will follow from 
\begin{equation} \label{eq:mass_chi}
\frac 1 { \mz^{2} } \lesssim \chi(z).
\end{equation}
For $z$ close to $z_c$, \eqref{eq:mass_chi} follows from $\chi(z) \ge (1- z/z_c)\inv$ and the asymptotic behaviour $\mz \asymp (z_c - z)^{1/2}$ from  \eqref{eq:mass_asymp}.
We therefore fix $z_1 < z_c$ and prove \eqref{eq:mass_chi} for $z \le z_1$. 
For these $z$, by the monotonicity of $m$ and $\chi$ we have
$\mz^{-2} =  \chi(0)\mz^{-2 } \le \chi(z) m(z_1)^{-2 }$.
This completes the proof of the upper bound of Theorem~\ref{thm:torus}.
\end{proof}

\subsection{Lower bound for the torus two-point function}
\label{sec:torus_lb}
To complete the proof of Theorem~\ref{thm:torus}, we need a lower bound of order $\chi(z)/r^d$ for the difference
\begin{equation}
\psi\subrz\supT (x) = G_z\supT(x) - G_z(x).
\end{equation}
We decompose $\psi\subrz\supT(x)$ as
\begin{equation} \label{eq:def_psi}
\psi\subrz\supT (x) = \psi\subrz(x) - ( \psi\subrz(x) - \psi\subrz\supT(x) )
\quad\text{with}\quad 
\psi\subrz(x) = \sum_{u\ne0} G_z(x+ru). 
\end{equation}
The next two lemmas estimate the two factors in the decomposition \eqref{eq:def_psi}. 

The first lemma concerns $\psi\subrz$ and is proved for wealky self-avoiding walk in \cite[(6.22)]{Slad23_wsaw}. 
The proof uses the asymptotics of $G\crit(x)$ given in \eqref{eq:G_crit_asymp} and a differential inequality, 
and it applies \emph{verbatim} to strictly self-avoiding walk. We therefore omit its proof. 

\begin{lemma} 
\label{lem:psi_lb}
Let $d>4$. 
There are constants $c_3, c_\psi > 0$ such that 
for all $r\ge 3$, $z \in [z_c - c_3r^{-2}, z_c )$, and $x\in \T^d_r$,
\begin{equation} \label{eq:psi}
\psi\subrz(x)  \ge  c_\psi \frac{ \chi(z) }{r^d }.
\end{equation}
\end{lemma}

The second lemma bounds the second term of \eqref{eq:def_psi}. 
Its proof generalises and corrects an inconsequential error in the proof given in \cite{Slad23_wsaw} for weakly self-avoiding walk.

\begin{lemma}
\label{lem:psi_diff}
Let $d>4$. 
There is a constant $C>0$ such that 
for all $r\ge 3$, $z \in (0, z_c)$, and $x\in \T^d_r$,
\begin{equation} \label{eq:psi_diff}
\psi\subrz(x) - \psi\subrz\supT(x) 
\le C \frac{ \chi(z) }{r^d } \( G_z*G_z(x) + \frac{ \chi(z)^2 }{r^d} \).
\end{equation}
\end{lemma}

\begin{proof}[Proof of lower bound of Theorem~\ref{thm:torus}, assuming Lemma~\ref{lem:psi_diff}]
The constant $c_3$ in the theorem is given by Lemma~\ref{lem:psi_lb}. To prove a lower bound for $\psi\subrz\supT(x)$, by decomposition \eqref{eq:def_psi} and Lemma~\ref{lem:psi_lb}, it suffices to show 
\begin{equation} \label{eq:psi_diff_2}
\psi\subrz(x) - \psi\subrz\supT(x) \le \half c_\psi \frac{ \chi(z) } {r^d},
\end{equation} 
where $c_\psi$ is the constant of \eqref{eq:psi}. 
This will yield $\psi\subrz\supT(x) \ge \half c_\psi \chi(z) / r^d$, which gives the desired result. 
To prove \eqref{eq:psi_diff_2}, we use Lemma~\ref{lem:psi_diff}. 
In view of \eqref{eq:psi_diff}, it suffices to choose $x$ and $z$ so that 
$G_z*G_z(x)$ and  ${ \chi(z)^2 } / {r^d}$ are sufficiently small. 

For $z$, since $\chi(z)\asymp (z_c-z)\inv$, 
we can make ${ \chi(z)^2 } / {r^d}$ small by choosing $(z_c-z)^2 r^d$ large, \ie, by choosing $z_c - z \ge c_4 r^{-d/2}$ with $c_4$ sufficiently large. 
For $x$, since $G_z(x) \le G\crit (x) \lesssim \nnnorm x^{-(d-2)}$ by \eqref{eq:G_crit_asymp}, a standard convolution estimate \cite[Proposition~1.7(i)]{HHS03} implies that $G_z*G_z(x) \lesssim \nnnorm x^{-(d-4)}$ independent of $z$,
so we can make $G_z*G_z(x)$ small by restricting to $\abs x \ge M$ with $M$ sufficiently large. 
These choices of $c_4$ and $M$ ensure \eqref{eq:psi_diff_2} holds, and the proof of the lower bound of Theorem~\ref{thm:torus} is complete.
\end{proof}

It remains to prove Lemma~\ref{lem:psi_diff}. 
By definitions of $\psi\subrz$ and $\psi\subrz\supT$, we have
\begin{align} \label{eq:psi_diff_rewrite}
\psi\subrz(x) - \psi\subrz\supT(x) 
&= \bigg( \sum_{u\ne0} G_z(x+ru) \bigg) - \( G_z\supT(x) - G_z(x) \)	 \nl
&= \bigg( \sum_{u\in \Z^d} G_z(x+ru) \bigg) -  G_z\supT(x). 
\end{align}
The difference comes from $\Z^d$ walks that gain additional self-intersections when projected onto the torus. Indeed, if a walk has the same number of self-intersections before and after the projection, then its contributions to $\sum_{u\in \Z^d} G_z(x+ru)$ and $G_z\supT(x)$ cancel exactly. 
More precisely, we let $\pi_r: \Z^d \to \T^d_r = (\Z/r\Z)^d$ denote the canonical projection map. 
For an $n$-step $\Z^d$ walk $\omega$ and $0\le s< t\le n$, we define
\begin{align}
U_\st (\omega) &= - \1 \{ \omega(s) = \omega(t) \},	\\
U_\st\supT (\omega) &= - \1 \{ \pi_r\omega(s) = \pi_r\omega(t) \},	\\
U_\st^+ (\omega) &= - \1 \{ \pi_r\omega(s) = \pi_r\omega(t) \text{ and }\omega(s) \ne \omega(t) \}	,
\end{align}
and define weights
\begin{align}
K(\omega) = \prod (1 +  U_\st(\omega)), \quad
K\supT(\omega) = \prod (1 +  U_\st\supT(\omega)), \quad
K^+(\omega) = \prod (1 +  U_\st^+(\omega)), 
\end{align}
where the products range over $0 \le s < t \le n$. 
For example, the product $K(\omega)$ evaluates to 0 if $\omega$ has at least one self-intersection, and it evaluates to 1 otherwise, so it gives precisely the weight of a $\Z^d$ self-avoiding walk. 
It follows from the definitions that $K\supT(\omega) = K(\omega) K^+(\omega)$, and that
\begin{alignat}{2}
G_z (y) &= \sum_{n=0}^\infty \sum_{\omega \in \Wcal_n(y) } z^nK(\omega)
	\qquad &&(y\in \Z^d), 	\\
G_z \supT (x) &= \sum_{u\in \Z^d} \sum_{n=0}^\infty \sum_{\omega \in \Wcal_n(x+ru) } z^n K\supT(\omega)
	\qquad &&(x\in \T^d_r). 
\end{alignat}
Plugging into \eqref{eq:psi_diff_rewrite}
and using $K(\omega) - K\supT(\omega) = K(\omega) [ 1 - K^+(\omega) ]$, we get
\begin{equation} \label{eq:psi_diff_lace}
\psi\subrz(x) - \psi\subrz\supT(x) 
= \sum_{u\in \Z^d} \sum_{n=0}^\infty \sum_{\omega \in \Wcal_n(x+ru) } z^n K(\omega) [ 1 - K^+(\omega) ]  .
\end{equation}
Note that if a walk $\omega$ gains no additional self-intersections when projected onto the torus, then $ 1 - K^+(\omega)=0$ and it does not contribute to $\psi\subrz - \psi\subrz\supT$. 
Equation \eqref{eq:psi_diff_lace} corrects \cite[(6.33)]{Slad23_wsaw} which misses the $u=0$ term. The error is not harmful, however, because all subsequent estimates in \cite{Slad23_wsaw} sum over all $u\in \Z^d$.

\begin{proof}[Proof of Lemma~\ref{lem:psi_diff}]
We use \eqref{eq:psi_diff_lace}. 
To bound $1 - K^+$, we use the inequality 
\begin{equation}
1 - \prod_{a\in A} (1 - u_a)	\le \sum_{a\in A} u_a 	\qquad (u_a\in[0,1]),
\end{equation}
which follows by induction on the cardinality of the index set $A$, 
with $u_a = \abs{ U_\st^+(\omega) }$. This gives the bound
\begin{equation}
\label{eq:psi_diff_bdd}
\psi\subrz(x) - \psi\subrz\supT(x) 
\le  \sum_{u\in \Z^d}  \sum_{n=0}^\infty \sum_{\omega \in \Wcal_n(x+ru) } z^n K(\omega) \sum_{0\le s<t \le n} \abs{ U_\st^+(\omega) }  .
\end{equation}
For the sum over $U_\st^+(\omega)$ to be nonzero, $\pi_r\omega$ must have more self-intersections than $\omega$. 
In this case, $\omega$ must pass through two distinct points $y$ and $y+rv$ ($v\in \Z^d$) that have the same torus projection. 
Without loss of generality, we assume $\omega$ travels from $0$ to $y$, $y$ to $y+rv$, then from $y+rv$ to $x+ru$. 
This decomposes $\omega$ into three subwalks $\omega_1 \in \mathcal W_{n_1}(y)$, $\omega_2 \in\mathcal W_{n_2}(rv)$, and $\omega_3 \in \mathcal W_{n_3}(x-y + r(u-v) )$, of lengths $n_1 = s$, $n_2 = t-s$, and $n_3 = n-t$, respectively. 
Using the decomposition, the sums over $n,s,t$ in \eqref{eq:psi_diff_bdd} is equivalent to summing over $n_1 \ge 0$, $n_2 \ge 1$, and $n_3\ge 0$. 
For the summand, we use $z^n = z^{n_1}z^{n_2}z^{n_3}$ and $K(\omega) \le K(\omega_1) K(\omega_2) K(\omega_3)$, which follows from disregarding the interactions between the subwalks. 
This yields 
\begin{align}
\label{eq:psi_diff_bdd_2}
\psi\subrz(x) - \psi\subrz\supT(x) 
&\le \sum_{u,y\in \Z^d}\sum_{v\ne0} G_z(y) G_z(rv) G_z(x-y+r(u-v))\nl
&= \sum_{v\ne0} G_z(rv) \sum_{w,y\in \Z^d} G_z(y) G_z(x-y+rw)  \nl
&= \psi_\rz(0) \Big( G_z * G_z(x) + \sum_{w\ne0} G_z* G_z (x+rw) \Big) ,
\end{align}
where $w=u-v$.
For $w\ne 0$, it follows from Theorem~\ref{thm:saw}, the inequality
$\abs y + \abs{ x-y+rw} \ge \abs{x+rw}$, and a standard convolution estimate \cite[Proposition~1.7(i)]{HHS03} that
\begin{equation}
G_z* G_z (x+rw) \lesssim
\sum_{y\in \Z^d} \frac{ e^{-c_1\mz \abs y} }{ \nnnorm y^{d-2} } 
	\frac{e^{-c_1\mz\abs{x-y+rw}} }{\nnnorm { x -y+rw }^{d-2} }
\lesssim \frac{ e^{-c_1\mz \abs{x+rw} } } { \nnnorm{x+rw}^{d-4} }, 
\end{equation}
so that, as in the proof of the upper bound of Theorem~\ref{thm:torus}, we can use \cite[Lemma~4.2]{HMS23} with $a=4$ and $\nu = c_1\mz$ to get
\begin{equation}
\label{eq:G_conv}
\sum_{w\ne0} G_z* G_z (x+rw)
\lesssim \frac 1 { \mz^4 r^d  } e^{-\frac {c_1} 4 \mz r }
\lesssim \frac{ \chi(z)^2 }{ r^d } e^{-\frac {c_1} 4 \mz r }
\le \frac{ \chi(z)^2 }{ r^d } .
\end{equation}
Since in the proof of upper bound of Theorem~\ref{thm:torus} we proved $\psi_\rz(x) \lesssim \chi(z) / r^d$ without using this notation, 
from \eqref{eq:psi_diff_bdd_2} and \eqref{eq:G_conv} we obtain
\begin{equation}
\psi\subrz(x) - \psi\subrz\supT(x) 
\lesssim \frac{ \chi(z) }{r^d} \( G_z * G_z(x) + \frac{ \chi(z)^2 }{ r^d } \), 
\end{equation}
which is the desired result.
\end{proof}

\subsection{Proof of Corollary~\ref{cor:torus}}
\label{sec:torus_cor}

\begin{proof}
We use Theorem~\ref{thm:torus}. 
The upper bound follows immediately from the upper bound of \eqref{eq:G^T} and $G_z(x) \lesssim \nnnorm x^{-(d-2)}$. 
The lower bound when $\abs x \ge M$ also follows from \eqref{eq:G^T}, 
provided 
\begin{equation} \label{eq:G_x}
G_z(x) \asymp \frac 1 {\nnnorm x^{d-2} }
\qquad (r\ge 3,\ z \in [z_c - c_3r^{-2}, z_c ), \ x\in \T^d_r ),
\end{equation}
which is proved for weakly self-avoiding walk in \cite[Remark~6.2]{Slad23_wsaw} (with a possibly smaller $c_3$) using a differential inequality. The proof applies \emph{verbatim} to strictly self-avoiding walk. 
For $\abs x \le M$, 
since $z \ge z_c - c_3r^{-2} \ge \frac{z_c}2$ when $r$ is sufficiently large,
and since there is a self-avoiding walk from $0$ to $x$ with length $\norm x _1 \le \sqrt d \abs x \le \sqrt d M$, we have the lower bound
\begin{equation}
G_z^\T(x) \ge (z \wedge 1 )^{\sqrt d M}
\ge \( \frac{z_c}2 \wedge 1 \)^{\sqrt d M} .
\end{equation} 
The desired result then follows from $\nnnorm x^{-(d-2)} \le 1$ and 
\begin{equation}
\frac{ \chi(z) }{ r^d }
\le \frac{ \chi(z_c - c_4 r^{-d/2}) }{ r^d }
\asymp \frac{ r^{d/2} }{ r^d  }\le 1.
\end{equation} 
This completes the proof of Corollary~\ref{cor:torus}.
\end{proof}

\section*{Acknowledgements}
The work was supported in part by NSERC of Canada.
We thank Gordon Slade for comments on a preliminary version.


{ \small
\begin{thebibliography}{10}

\bibitem{BBS-saw4}
R.~Bauerschmidt, D.C. Brydges, and G.~Slade.
\newblock Critical two-point function of the 4-dimensional weakly self-avoiding
  walk.
\newblock {\em Commun.\ Math.\ Phys.}, {\bf 338}:169--193, (2015).

\bibitem{BCHSS05a}
C.~Borgs, J.T. Chayes, R.~van~der Hofstad, G.~Slade, and J.~Spencer.
\newblock Random subgraphs of finite graphs: {I}. {The} scaling window under
  the triangle condition.
\newblock {\em Random Struct. Alg.}, {\bf 27}:137--184, (2005).

\bibitem{BCHSS05b}
C.~Borgs, J.T. Chayes, R.~van~der Hofstad, G.~Slade, and J.~Spencer.
\newblock Random subgraphs of finite graphs: {II}. {The} lace expansion and the
  triangle condition.
\newblock {\em Ann. Probab.}, {\bf 33}:1886--1944, (2005).

\bibitem{BHH21}
D.C. Brydges, T.~Helmuth, and M.~Holmes.
\newblock The continuous-time lace expansion.
\newblock {\em Commun. Pure Appl. Math.}, {\bf 74}:2251--2309, (2021).

\bibitem{BS85}
D.C. Brydges and T.~Spencer.
\newblock Self-avoiding walk in 5 or more dimensions.
\newblock {\em Commun. Math. Phys.}, {\bf 97}:125--148, (1985).

\bibitem{CC86b}
J.T. Chayes and L.~Chayes.
\newblock {Ornstein-Zernike} behavior for self-avoiding walks at all
  noncritical temperatures.
\newblock {\em Commun. Math. Phys.}, {\bf 105}:221--238, (1986).

\bibitem{DGGZ24}
Y.~Deng, T.M. Garoni, J.~Grimm, and Z.~Zhou.
\newblock Two-point functions of random-length random walk on high-dimensional
  boxes.
\newblock {\em J. Stat. Mech: Theory Exp.}, 023203, (2024).

\bibitem{DP25-Ising}
H.~Duminil-Copin and R.~Panis.
\newblock New lower bounds for the (near) critical {Ising} and $\varphi^4$
  models' two-point functions.
\newblock {\em Commun. Math. Phys.}, {\bf 406}:56, (2025).

\bibitem{FH17}
R.~Fitzner and R.~van~der Hofstad.
\newblock Mean-field behavior for nearest-neighbor percolation in $d>10$.
\newblock {\em Electron. J. Probab.}, {\bf 22}:1--65, (2017).

\bibitem{FH21}
R.~Fitzner and R.~van~der Hofstad.
\newblock {N}o{B}{L}{E} for lattice trees and lattice animals.
\newblock {\em J. Stat. Phys.}, {\bf 185}:paper 13, (2021).

\bibitem{Graf14}
L.~Grafakos.
\newblock {\em Classical Fourier Analysis}.
\newblock Springer, New York, 3rd edition, (2014).

\bibitem{Hara08}
T.~Hara.
\newblock Decay of correlations in nearest-neighbor self-avoiding walk,
  percolation, lattice trees and animals.
\newblock {\em Ann. Probab.}, {\bf 36}:530--593, (2008).

\bibitem{HHS03}
T.~Hara, R.~van~der Hofstad, and G.~Slade.
\newblock Critical two-point functions and the lace expansion for spread-out
  high-dimensional percolation and related models.
\newblock {\em Ann. Probab.}, {\bf 31}:349--408, (2003).

\bibitem{HS90a}
T.~Hara and G.~Slade.
\newblock Mean-field critical behaviour for percolation in high dimensions.
\newblock {\em Commun. Math. Phys.}, {\bf 128}:333--391, (1990).

\bibitem{HS90b}
T.~Hara and G.~Slade.
\newblock On the upper critical dimension of lattice trees and lattice animals.
\newblock {\em J. Stat. Phys.}, {\bf 59}:1469--1510, (1990).

\bibitem{HS92b}
T.~Hara and G.~Slade.
\newblock The lace expansion for self-avoiding walk in five or more dimensions.
\newblock {\em Rev. Math.\ Phys.}, {\bf 4}:235--327, (1992).

\bibitem{HS92a}
T.~Hara and G.~Slade.
\newblock Self-avoiding walk in five or more dimensions. {I.} {The} critical
  behaviour.
\newblock {\em Commun.\ Math.\ Phys.}, {\bf 147}:101--136, (1992).

\bibitem{HMS23}
T.~Hutchcroft, E.~Michta, and G.~Slade.
\newblock High-dimensional near-critical percolation and the torus plateau.
\newblock {\em Ann. Probab.}, {\bf 51}:580--625, (2023).

\bibitem{LL10}
G.F. Lawler and V.~Limic.
\newblock {\em Random Walk: A Modern Introduction}.
\newblock Cambridge University Press, Cambridge, (2010).

\bibitem{Liu25_thesis}
Y.~Liu.
\newblock {\em Critical and near-critical scaling in high-dimensional
  statistical mechanics}.
\newblock PhD thesis, University of British Columbia, (2025).

\bibitem{LPS25-universal}
Y.~Liu, J.~Park, and G.~Slade.
\newblock Universal finite-size scaling in high-dimensional critical phenomena.
\newblock Preprint, \url{https://arxiv.org/pdf/2412.08814}, (2024).

\bibitem{LS24b}
Y.~Liu and G.~Slade.
\newblock Gaussian deconvolution and the lace expansion for spread-out models.
\newblock To appear in {\it Ann.\ Inst.\ H.\ Poincar\'e Probab.\ Statist.}
  Preprint, \url{https://arxiv.org/pdf/2310.07640}, (2023).

\bibitem{LS24a}
Y.~Liu and G.~Slade.
\newblock Gaussian deconvolution and the lace expansion.
\newblock {\em Probab.\ Theory Related Fields}, (2024).
\newblock \url{https://doi.org/10.1007/s00440-024-01350-9}.

\bibitem{LS25a}
Y.~Liu and G.~Slade.
\newblock Near-critical and finite-size scaling for high-dimensional lattice
  trees and animals.
\newblock {\em J. Stat. Phys.}, {\bf 192}:article 23, (2025).

\bibitem{MS93}
N.~Madras and G.~Slade.
\newblock {\em The Self-Avoiding Walk}.
\newblock Birkh{\"a}user, Boston, (1993).

\bibitem{Mich23}
E.~Michta.
\newblock The scaling limit of the weakly self-avoiding walk on a
  high-dimensional torus.
\newblock {\em Electron. Commun. Probab.}, {\bf 28}:1--13, (2023).

\bibitem{MPS23}
E.~Michta, J.~Park, and G.~Slade.
\newblock Boundary conditions and universal finite-size scaling for the
  hierarchical $|\varphi|^4$ model in dimensions $4$ and higher.
\newblock {\em Commun. Pure Appl. Math.}, (2025).
\newblock \url{https://doi.org/10.1002/cpa.22256}.

\bibitem{MS22}
E.~Michta and G.~Slade.
\newblock Asymptotic behaviour of the lattice {Green} function.
\newblock {\em ALEA, Lat. Am. J. Probab. Math. Stat.}, {\bf 19}:957--981,
  (2022).

\bibitem{MS23}
E.~Michta and G.~Slade.
\newblock Weakly self-avoiding walk on a high-dimensional torus.
\newblock {\em Probab. Math. Phys.}, {\bf 4}:331--375, (2023).

\bibitem{Saka07}
A.~Sakai.
\newblock Lace expansion for the {Ising} model.
\newblock {\em Commun. Math. Phys.}, {\bf 272}:283--344, (2007).
\newblock Correction: A.~Sakai. Correct bounds on the Ising lace-expansion
  coefficients. {\it Commun. Math. Phys.}, {\bf 392}:783--823, (2022).

\bibitem{Saka15}
A.~Sakai.
\newblock Application of the lace expansion to the $\varphi^4$ model.
\newblock {\em Commun. Math. Phys.}, {\bf 336}:619--648, (2015).

\bibitem{Slad06}
G.~Slade.
\newblock {\em The Lace Expansion and its Applications.}
\newblock Springer, Berlin, (2006).
\newblock Lecture Notes in Mathematics Vol. 1879. Ecole d'Et\'{e} de
  Probabilit\'{e}s de Saint--Flour XXXIV--2004.

\bibitem{Slad22_lace}
G.~Slade.
\newblock A simple convergence proof for the lace expansion.
\newblock {\em Ann.\ I.\ Henri Poincar\'{e} Probab.\ Statist.}, {\bf
  58}:26--33, (2022).

\bibitem{Slad23_wsaw}
G.~Slade.
\newblock The near-critical two-point function and the torus plateau for weakly
  self-avoiding walk in high dimensions.
\newblock {\em Math. Phys. Anal. Geom.}, {\bf 26}:article 6, (2023).

\bibitem{ZGFDG18}
Z.~Zhou, J.~Grimm, S.~Fang, Y.~Deng, and T.M. Garoni.
\newblock Random-length random walks and finite-size scaling in high
  dimensions.
\newblock {\em Phys. Rev. Lett.}, {\bf 121}:185701, (2018).

\end{thebibliography}
}
\end{document}